\documentclass[12pt]{article}
\usepackage{e-jc}

\usepackage{amsmath,graphicx}
\usepackage[USenglish]{babel}
\usepackage[latin1]{inputenc}
\usepackage{epsfig}
\usepackage{epic,eepic}
\usepackage{rotating}
\usepackage{verbatim}
\usepackage{enumerate}
\usepackage{tikz}
\usepackage[colorlinks=true]{hyperref}
\usepackage{cleveref}

\usepackage{booktabs}
\usepackage{diagbox}

\crefname{counterexample}{counterexample}{counterexamples}
\Crefname{counterexample}{Counterexample}{Counterexamples}
\crefname{conjecture}{conjecture}{conjectures}
\Crefname{conjecture}{Conjecture}{Conjectures}

\numberwithin{theorem}{section}

\newcommand{\lk}{\operatorname{lk}}
\newcommand{\rk}{\operatorname{rank}}
\newcommand{\F}{\mathbb F}

\newcommand{\N}{\mathbb N}

\newcommand{\Vertices}{ V}
\newcommand{\sep}{\tau}
\newcommand{\uu}{\underline}
\newcommand{\sphere}{S}
\newcommand{\altsphere}{T}

\newcommand{\manifold}{M}
\newcommand{\altmanifold}{N}

\dateline{March 8, 2019}{May 3, 2020}{TBD}


\Copyright{The authors. Released under the CC BY-ND license (International 4.0).}

\title{Average Betti numbers of induced subcomplexes \\ in triangulations of manifolds%
\thanks{The three authors were supported by grant EVF-2015-230 of the Einstein Foundation Berlin and, while they were in residence at the Mathematical Sciences Research Institute in Berkeley, California during the Fall 2017 semester, by the Clay Institute and the National Science Foundation (Grant No. DMS-1440140). Santos is also supported by grants MTM2014-54207-P and MTM2017-83750-P of the Spanish Ministry of Science.}
}

\author{Giulia Codenotti \qquad Jonathan Spreer\\
\small Institut f\"ur Mathematik\\[-0.8ex]
\small Freie Universit\"at Berlin, Germany\\
\small\tt codenotti@math.uni-frankfurt.de\qquad
\tt jonathan.spreer@sydney.edu.au\\
\and
Francisco Santos\\
\small Department of Mathematics, Statistics and Computer Science\\[-0.8ex]
\small University of Cantabria, Santander, Spain\\
\small\tt francisco.santos@unican.es
}

\begin{document}

\maketitle

\begin{abstract}
We study a variation of Bagchi and Datta's $\sigma$-vector of a simplicial complex $C$, whose entries are defined as weighted averages of Betti numbers of induced subcomplexes of $C$. We show that these invariants satisfy an Alexander-Dehn-Sommerville type identity, and behave nicely under natural operations on triangulated manifolds and spheres such as connected sums and bistellar flips.

In the language of commutative algebra, the invariants are weighted sums of graded Betti numbers of the Stanley-Reisner ring of $C$. This interpretation implies, by a result of Adiprasito, that the Billera-Lee sphere maximizes these invariants among triangulated spheres with a given $f$-vector. 
For the first entry of $\sigma$, we extend this bound to the class of strongly connected pure complexes.

  As an application, we show how upper bounds on $\sigma$ can be used to obtain lower bounds on the $f$-vector of triangulated $4$-manifolds with transitive symmetry on vertices and prescribed vector of Betti numbers.

\end{abstract}

\section{Introduction}
In this article we investigate a combinatorial invariant of simplicial complexes that we call the $\sep$-vector, defined as follows: for a simplicial complex $C$ with ground set $V$, and for each $i=\{-1,0,1,2,\dots\}$, 
\begin{equation*}
{\sep}_i(C)= \frac1{|V|+1}\sum_{W\subseteq \Vertices} \frac{\tilde{\beta_i}(C[W])}{\binom{|V|}{|W|}}.
\end{equation*}
Here $C[W]$ denotes the subcomplex induced by a set $W\subseteq V$, and $\tilde \beta_i$ is the reduced $i$-th Betti number with respect to a certain field $\F$. The $\sep$-vector depends on the choice of $\F$ but our results are independent of $\F$.

Put differently, $\sep_i$ is the weighted average of the $i$-th reduced Betti number of all induced subcomplexes $C[W]$, with respect to weights that are uniform on subsets $W\subseteq V$ of equal size and add up to $\frac1{|V|+1}$ for each size $j\in \{0,\dots, |V|$\}.

The purpose of this article is twofold. On the one hand we give a comprehensive overview of research done on the $\sep$-vector, thereby unifying notation and language. On the other hand we add new results about this combinatorial invariant simplifying its study.

\subsubsection*{Previous work} The $\sep$-vector is called the \emph{normalized $\sigma$-vector} by Murai and Novik in \cite{Novik17NormalizedSigma}, where it is denoted $\tilde \sigma$. It is a variation of the \emph{$\sigma$-vector} introduced by Bagchi and Datta \cite{Bagchi14StellSpheresTightness} and studied in~\cite{Bha2016,Burton14SepIndex2Spheres}. Its original motivation was the study of \emph{tight} triangulations of manifolds, that is, triangulations with the property that all the homomorphisms induced in homology by inclusions of induced subcomplexes are injective (see \Cref{ssec:tightness}, in particular \Cref{defi:tight} for more details):

\begin{theorem}[\protect{\cite[Theorem 1.8(c), Corollary 1.9]{Bha2016}}; \protect{\cite[Theorems 2.6(b) and 2.10]{Bagchi14StellSpheresTightness}} for the $2$-neighborly case] 
 \label{thm:bagchi-intro}
For a simplicial $d$-manifold $M$ with vertex set $V$ and for every $i\in \{0,\dots,d\}$, let 
\[
\mu_i(M):=\sum_{v\in V} \sep_{i-1}(\lk_M(v)).
\] 
Then $\tilde\beta_i(M) \le \mu_i(M)$, with equality occurring for all $i$ if and only if $M$ is $\F$-tight.
\end{theorem}

Moreover, the $\sep$-vector is also interesting from a more combinatorial viewpoint. For instance, it produces the following characterization of stacked spheres among normal pseudo-manifolds (see~\Cref{ssec:stacked} for more details):

\begin{theorem}[Murai~\protect{\cite[Corollary 5.8.(ii)]{Murai15GradedBettiNumbers}} for the general case, 
Burton, Datta, Singh and Spreer \protect{\cite[Theorem 1.1]{Burton14SepIndex2Spheres}} for $2$-spheres]
\label{thm:BDSS}
Let $\sphere$ be a triangulated normal pseudo-manifold of dimension $d\ge 2$ with $n$ vertices. Then 
\[
\sep_0(\sphere) \le \frac{\binom{n-d-1}{2}}{(n+1)\binom{d+3}{2}},
\]
with equality if and only if $\sphere$ is a stacked $d$-sphere.
\end{theorem}

This result appears also as Lemma 7.2 in \cite{Novik17NormalizedSigma}. Lemma 4.3 in \cite{Murai17FaceNumbersFundGroup} implies a version of it for relative complexes.
Bagchi in~\cite{Bha2016} conjectures a similar bound valid for all $0\leq i \leq d$, which has~\Cref{thm:BDSS} as the case $i=0$, and proves it for certain spheres he refers to as {\em tame}.

\medskip

The $\sep$-vector also has a commutative algebra interpretation as observed in~\cite{Murai15GradedBettiNumbers}. Recall that to a simplicial complex $C$ one associates its Stanley-Reisner ideal $I(C)$. The minimal resolution of $I(C)$ gives rise to a triangular array of \emph{graded Betti numbers} $r_{i,j}$ for $-1 \le i < j \le |V|$. Hochster's formula says that each $r_{i,j}$ equals the sum of $\tilde\beta_{j-i-2} (C[W])$ ranging over all subsets $W\subseteq V$ of size $j$ (see~\Cref{eq:Hochster}). 

In particular, the entries of the $\sep$-vector are non-negative linear combinations of the graded Betti numbers, which implies that known upper and lower bounds for the latter apply to the former. It was shown by Migliore and Nagel~\cite{MiglioreNagel03Polytopes,Nagel08EmptySimplices} that, among all the simplicial \emph{polytopal} spheres with a given $f$-vector, the Billera-Lee spheres~\cite{Billera81ProofSuffMcMullenCondFVec} maximize every $r_{i,j}$, and hence also $\tau_i$, over any field of characteristic zero; see also \cite[Sect.~3]{Murai17FaceNumbersFundGroup}. In a preprint version of this paper we conjectured this statement for arbitrary triangulated spheres. This was later proven by Adiprasito for fields of characteristic zero as a by-product of his, as yet unpublished, proof of the $g$-theorem for spheres \cite{Adiprasito18Lefschetz}. 

 \begin{theorem}[Adiprasito \protect{\cite[Section 1.6]{Adiprasito18Lefschetz}}]
      \label{conj:general-intro}
      Let $\sphere$ be a $d$-sphere and let $\altsphere$ be the Billera-Lee $d$-sphere with $f$-vector $f(\sphere)$. Then $r_{i,j}(\sphere) \leq r_{i,j}(\altsphere)$ for all $i$, where $r_{i,j}$ is computed with respect to a field of characteristic zero. As a consequence, $\sep_i(\sphere) \leq \sep_i(\altsphere)$. 
    \end{theorem}

\subsubsection*{Outline of the paper}
In \Cref{sec:chordal} we look at $\sep_0$ for the class of pure complexes and prove a version of \Cref{conj:general-intro} in this context: Among all pure strongly connected complexes with a given dimension $d\ge 3$ and numbers of vertices and edges $(f_0,f_1)$ the Billera-Lee balls with those parameters (which exist since a strongly connected $d$-manifold has $f_1\in [df_0-\binom{d+1}2, \binom{f_0}2]$) maximize $\sep_0$ (\Cref{cor:billeralee} and \Cref{thm:upperbound}). The proof is elementary and uses the fact that the $\sep_0$ value, an invariant of the $1$-skeleton, can be upper bounded in terms of the in-degree sequence associated to an ordering of the vertices, with equality in the upper bound if and only if the ordering is a \emph{partial elimination order} (\Cref{lemma:peo}).

The differences between the $\sep$- and the original $\sigma$-vector from~\cite{Bha2016,Bagchi14StellSpheresTightness} are a factor of ${|V|+1}$ in the denominator and that the case $W=\emptyset$ is treated differently than in \cite{Bha2016,Bagchi14StellSpheresTightness}. These two (minor) differences make the definition more natural. For instance, they simplify \Cref{thm:bagchi-intro} (compare our definition of $\mu$ in \Cref{thm:bagchi-intro} to \cite[Definition 1.6]{Bha2016} and \cite[Definition 2.1]{Bagchi14StellSpheresTightness}, where a distinction needs to be made for the cases $i=0,1$).

More importantly, the $\sep$-vector is ``independent of the ground set'', as already shown by Murai and Novik in \cite[Lemma 4.4]{Novik17NormalizedSigma} (we offer a proof in \Cref{thm:ground_set}).
By this we mean the following; suppose that we have a complex $C$ with vertex set $V_0$ but which can also be considered as a complex on a bigger ground set $V \supsetneq V_0$. This happens naturally for each link $\lk_M(v)$ appearing in \Cref{thm:bagchi-intro}, where $V_0$ is the set of vertices adjacent to $v$ in the $1$-skeleton of $M$. Then it makes sense to calculate the $\sep$- or $\sigma$-vector of this complex both with respect to $V$ and with respect to $V_0$. The $\sep$-vector is independent of this choice (but the same is not true for the original $\sigma$-vector). 
The underlying reason for this nice property is that our normalization causes the weights used on $2^V$ to define an ``exchangeable probability measure'' (see \Cref{rem:probability}).

This independence of the ground set simplifies computations in several places. For example, in the context of  connected sums it has the consequence that one can apply Mayer-Vietoris type sequences to all induced subcomplexes simultaneously, to easily obtain the following statement in \Cref{ssec:connsum}.

\begin{theorem}[\Cref{thm:connectedsum}]
\label{thm:connectedsum-intro}
Let $\manifold_1$ and $\manifold_2$ be  simplicial $d$-manifolds, $d\ge 2$. Then
\begin{align*}
 {\sep}_i(\manifold_1 \# \manifold_2) &= \sep_i(\manifold_1) + \sep_i(\manifold_2), & i\in \{1,\dots,d-2\}  \qquad \text{ and }\\
  {\sep}_j(\manifold_1 \# \manifold_2) &=  \sep_j(\manifold_1) + \sep_j(\manifold_2) +  c(d,n_1,n_2) ,
 & j \in \{0,d-1\},
\end{align*}
where $n_i=f_0(\manifold_i)$ is the number of vertices in $\manifold_i$ and
\[
c(d,n_1,n_2):= \frac{1}{d+2} - \frac{1}{n_1 +1}- \frac{1}{n_2 +1} +\frac{1}{n_1+n_2-d-1}.
\]
\end{theorem}

In the special case where $M_2$ is the boundary of a $(d+1)$-simplex, the change from $M_1$ to $M$ is a so-called \emph{bistellar flip of type $(1,d+1)$}. In particular,  \Cref{thm:connectedsum-intro} yields a formula for how the $\sep$-vector changes under such flips (\Cref{cor:stacking}) and for the $\sep$-vector of stacked spheres, which are the simplicial spheres that can be obtained form the boundary of a simplex via such flips. 
For the more general case of $(i,j)$-flips in manifolds there is no closed formula on how the $\sep$-vector changes, but some partial results can be stated, see~\Cref{ssec:bistflips}.

For neighborly and stacked manifolds some components of the $\sep$-vector become zero. More precisely, for a $k$-neighborly complex all $\sep_i$ with $i\le k-2$ are zero (see \Cref{prop:neighborly}) and for a $k$-stacked $d$-sphere all $\sep_i$ with $k \leq i \leq d-k-1$ are zero (see \Cref{thm:kstacked}). The first result is due to Bagchi and Datta \cite[Lemma 3.9]{Bagchi14StellSpheresTightness}, the second is due to Bagchi \cite[Lemma 3]{Bagchi15TightnessCrit}.
Going further, we also look at manifolds that are almost $2$-neighborly or almost $1$-stacked and prove bounds for the entries of their $\sep$-vectors (\Cref{thm:nearlyneighb,thm:g21_sep0,thm:g21_sep1}). In particular, \Cref{ssec:nstacked} completely characterizes the possible $\sep$-vectors of spheres with $g_2=1$, following a classification of such spheres by Nevo and Novinsky \cite{NevoNovinsky11}. 

For arbitrary triangulated $3$-spheres, the topological Alexander duality implies that the $\sep$-vector is symmetric, i.e., that we have $\sep_i (\sphere) = \sep_{d-i-1} (\sphere)$. Together with the Dehn-Sommerville relations this yields a formula relating the first half of the $\sep$- and $h$-vector to each other (see \Cref{cor:oddsphere}). For example, for $3$-spheres Alexander duality implies that any of $\sep_0, \sep_1, \sep_2$ (together with the $f$-vector or, equivalently, the $h$-vector) gives the other two, via the following relations (\Cref{cor:3d}):
\begin{equation}
\label{eq:sep_for_3spheres}
\sep_0=\sep_2, \quad  {\sep}_0 - {\sep_{1}} + \sep_2  = \frac{h_1(h_1+1)}{10(h_1+5)}  - \frac{h_2}{30}.
\end{equation}

In particular, this implies that the $\sep$-vector of a triangulated $3$-sphere is determined by $\sep_0$, and thus can be obtained purely combinatorially by counting connected components of induced subcomplexes.

Finally, moving towards more topological applications of the $\sep$-vector, we use the upper bound for $\sep$ in dimension three to construct lower bounds for triangulations of $4$-manifolds with vertex transitive symmetry and prescribed vector of Betti numbers (see \Cref{ssec:4mflds}). These lower bounds are interesting because, when attained, they not only provide a minimal symmetric triangulation, but also one that is tight. Moreover, we present exhaustive computations for spheres with up to ten points in dimension three.
We also present an example of a tight triangulation of a $4$-manifold whose vertex links do not have maximal $\sep$-vectors amongst all $3$-spheres with a given $f$-vector (see \Cref{ex:tight}). This is surprising in light of the Lutz-K\"uhnel conjecture, see \Cref{conj:LutzKuehnel} and \cite{Kuehnel99CensusTight}.

\subsubsection*{Acknowledgements}
We thank Karim Adiprasito, Bhaskar Bagchi, Satoshi Murai, Isabella Novik, Jos\'e Alejandro Samper and Hailun Zheng for useful comments on a previous version of this paper. We also thank Moritz Firsching for helpful discussions regarding the computations in Section 6.

\section{Setup and first properties of $\sep$}

\subsection{Notation and conventions}
\label{ssec:notation}

In this section we briefly introduce most of the combinatorial/topological concepts used all throughout the paper and set  notation and conventions.

\subsubsection*{Simplicial complexes}

An (abstract) \emph{simplicial complex} $C$ on a ground set $V$ (typically $V=[n]$ for an $n\in \N$) is a family of subsets of $V$ closed under taking subsets.  
The elements of $C$ are called \emph{faces}, and the \emph{dimension} of a face is its size minus one. Maximal faces are called \emph{facets} and faces of dimensions 0 and  1 are called \emph{vertices} and \emph{edges}, respectively. 
Observe that the set of vertices can be properly contained in $V$, that is, not every element of $V$ is necessarily a vertex of $C$.
 The \emph{empty complex} has only one face, the empty set. 
 The $k$-skeleton of $C$ is the subcomplex consisting of faces of dimension at most $k$. 
 The dimension of $C$ is the maximum dimension of a face, and a complex of dimension $d$ is sometimes called a \emph{$d$-complex}.

The following properties that a simplicial complex may have are each more restrictive than the previous one:

\begin{itemize}
\item $C$ is \emph{pure} of dimension $d$ if all facets have dimension $d$. In this case the faces of dimension $(d-1)$ are called \emph{ridges} and the \emph{adjacency graph} of $C$, sometimes also referred to as the {\em dual graph} of $C$, is defined as having as vertices the facets of $C$ and as edges the pairs of facets that share a common ridge. This is different from the \emph{$1$-skeleton} of $C$, which is also a graph and sometimes called the \emph{graph} of $C$.

\item A pure complex $C$ is called \emph{strongly connected} if its adjacency graph is connected. 

\item A \emph{(closed) pseudo-manifold} is a strongly-connected pure complex such that every ridge is contained in exactly two facets. In this case there exists a bijective correspondence between ridges of $C$ and edges in the adjacency graph. A \emph{pseudo-manifold with boundary} is a strongly-connected pure complex such that every ridge is contained in at most two facets. The ridges contained in only one facet, together with all their faces, form the \emph{boundary} of $C$, which is a pure $(d-1)$-complex.

\item A \emph{(closed) triangulated manifold}, or simply a \emph{(closed) manifold}, is a simplicial complex whose topological realization is a (closed) manifold. We say the manifold is a sphere, a ball, etc if its topological realization is.

\item  A \emph{combinatorial manifold} is a simplicial complex in which the {\em link} of every vertex $v$, i.e., the boundary of the complex consisting of all facets of $C$ containing $v$ and their faces, is a triangulated sphere PL-homeomorphic to the standard sphere.
\end{itemize}

\subsubsection*{$f$-, $h$-, and $g$-vectors}
The $f$-vector of a  $d$-complex is defined as $f(C)=(f_{-1},\dots,f_{d})$ where $f_i$ is the number of faces of size $i+1$. 
For a pure $d$-complex $C$ with $f$-vector $(f_{-1},f_0, \ldots , f_d)$ one defines the \emph{$h$-vector} $h(C)=(h_0,\dots, h_{d+1})$ by
  \[ h_{k}=\sum _{i=0}^{k} (-1)^{k-i} {\binom {d+1-i}{k-i}} f_{i-1}.\]
The $h$-vector contains the same information as the $f$-vector since the above formula can be reversed, but it has nicer properties; for example, the $h$-vector of a closed manifold satisfies the Klee-Dehn-Sommerville equations \cite{Klee64CombAnaloguePoincareDual}:
\[
h_{d+1-i} - h_i= (-1)^i \binom{d+1}{i} \left(\chi(C) - \chi(\sphere^{d})\right) \quad \text{for all } i.
\]
From the $h$-vector one can, in turn, define the $g$-vector $g(C)=(g_0,\dots, g_{d+1})$ as $g_i=h_i-h_{i-1}$. See \cite[Chapter 17]{HandbookDCG} for more details.

\subsubsection*{Betti numbers}

For any given base field $\F$ the \emph{reduced chain complex} of $C$ is the naturally defined sequence of linear maps
\[
\F^{f_d} \stackrel{\partial_d}{\to} \F^{f_{d-1}} \to\dots \to \F^{f_1}  \stackrel{\partial_1}{\to} \F^{f_0}  \stackrel{\partial_0}{\to} \F^{f_{-1}} = \F \to 0.
\]
The \emph{reduced Betti numbers} $\tilde \beta_d, \dots, \tilde \beta_0, \tilde \beta_{-1}$ of $C$ are the dimensions of the corresponding homology (or cohomology) groups. Equivalently, 
\[
\tilde \beta_i := f_i -\rk \partial_i - \rk \partial_{i+1}.
\]
Thus, the alternating sum of Betti numbers coincides with the alternating sum of $f$-vector entries; the \emph{Euler characteristic}  of $C$ equals that sum plus one:
\[
\chi(C) := \sum_{i=-1}^d (-1)^i \tilde\beta_i +1 = \sum_{i=-1}^i  (-1)^i f_i +1.
\]

\begin{remark} 
Observe that every complex has $f_{-1}=1$, for the empty face; in particular $\chi(C) =  \sum_{i=0}^i  (-1)^i f_i$.
\end{remark}

\subsubsection*{(Induced) subcomplexes, Alexander duality, Hochster's formula}

Let $C$ be a simplicial complex on a ground set $V$. A subcomplex of $C$ is a subset of the faces of $C$ that is itself a simplicial complex. For any $W\subseteq V$, the subcomplex of $C$ \emph{induced by $W$} is
\[
C[W] := \{F \in C : F \subseteq W\}
\]
and the \emph{deletion of $W$ in $C$} is the induced subcomplex $C[V\setminus W]$.
When $C$ is a $d$-sphere, the Betti numbers of these two subcomplexes are related by  \emph{Alexander duality}:
\[
\tilde \beta_i C[W] = \tilde \beta_{d-1-i} C[V\setminus W].
\]

\begin{remark}
In this paper we always mean Alexander duality in the above, topological, sense. This is related but not to be confused with the combinatorial version of Alexander duality as, for instance, stated in \cite{Bjoerner09CombAlexDuality}.
\end{remark}

Identifying the elements of $V$ with variables $x_1, \ldots , x_n$, the \emph{Stanley-Reisner} ideal $I_C$ of $C$ is the ideal in $S:=\F[x_1,\dots,x_n]$ generated by all monomials whose support is not a face in $C$. It has a minimal resolution which is unique modulo isomorphism
\[
S^{r_n} \to S^{r_{n-1}} \to \dots \to S^{r_0} \to I_C \to 0.
\]
The ranks $r_i$ appearing in the resolution can be refined as follows: The usual grading in $I_C$ induces a grading in all the $S^{r_{i}}$'s, so that we can write
\[
S^{r_{i}} = \bigoplus_{j=i+1}^n S^{r_{i,j}},
\] 
where $S^{r_{i,j}}$ is the part of $S^{r_{i}}$ of degree $j$.
By convention we take $r_{-1,0}= 1$ and $r_{-1,j}= 0$ for $j\ne 0$.
(The convention for $i=-1$ is justified by looking at the resolution of $S/I$ rather than $I$. That resolution ends in $\to S \to S/I \to 0$, where the generator in $S= S^{r_{-1}}$ has degree zero.)
The numbers $r_{i,j}=r_{i,j}(I_C)$ for $-1 \le i < j \le n$ are called the \emph{graded Betti numbers} of $I_C$. 

Hochster's formula~\cite[Corollary 5.12]{MillerSturmfels05CCA} implies that
\begin{equation}
\label{eq:Hochster}
r_{i,j}(I_C) = \sum_{\substack{W\subseteq \Vertices\\ |W|=j}} \tilde\beta_{j-i-2} (C[W]).
\end{equation}
In particular, $r_{i,n}(I_C) = \tilde\beta_{n-i-2}(C)$ and $r_{0,j}(C)$ equals the number of minimal non-faces of size $j$ in $C$, since these correspond to the generators of degree $j$ in $I_C$.

\subsection{The $\sep$-vector} 

In \cite{Bagchi14StellSpheresTightness}, Bagchi and Datta introduce the \emph{$\sigma$-vector} of a simplicial complex $C$. It is defined as follows:

\begin{definition}[$\sigma$-vector, Definition 2.1 in \cite{Bagchi14StellSpheresTightness}]
\label{defi:sigma}
  Let $C$ be a simplicial complex of dimension $d$ with vertex set $\Vertices$. For each $i=-1,\dots,d$ we define
  \[
    \sigma_i(C) = \sum_{W\subseteq \Vertices} \frac{\tilde\beta_i(C[W])}{{ |V| \choose |W|}}. 
  \]
Here, $\tilde\beta_i$ denotes the $i$-th reduced Betti number with respect to a certain field $\F$, that we omit. 
\end{definition}

In this article, we slightly adapt the definition of the $\sigma$-vector:

\begin{definition}
Let $C$ be a simplicial complex of dimension $d$ on a ground set $\Vertices$ of size $|\Vertices|=n$. 
For each integer $i\ge -1$ we define 
\begin{equation*}
{\sep}_i(C)= \frac1{n+1}\sum_{W\subseteq \Vertices} \frac{\tilde{\beta_i}(C[W])}{\binom{n}{|W|}}.
\end{equation*}
\end{definition}

\begin{remark}
\label{rem:correct_beta}
In~\cite{Bha2016}, \cite{Bagchi14StellSpheresTightness} and \cite{Burton14SepIndex2Spheres},  $\tilde\beta_0$ is defined to be equal to $\beta_0 -1$. This coincides with our convention, see \Cref{ssec:notation}, except for the empty complex where their convention gives  $\tilde \beta_0 = -1$ and ours gives $\tilde \beta_0 = 0$. We believe our convention is more commonly used, see for instance \cite{Hatcher2002AlgTop}. It also behaves more nicely, e.g., regarding Alexander duality. All statements from the above papers that we cite in this article have been adapted to our convention.
\end{remark}

Hochster's~\Cref{eq:Hochster} gives the following interpretation of the $\sep$-vector:

\begin{proposition}
\label{prop:Hochster}
Let $C$ be a simplicial complex on a vertex set $V$ of size $n$ and let $\{r_{i,j}\}_{0\le i \le j \le n}$ be the graded Betti numbers of the Stanley-Reisner ideal $I_C$. Then
\[
\sep_i = 
\frac1{n+1}\sum_{j=i+1}^n \frac{r_{j-i-2,j}}{\binom{n}{j}}.
\]
\end{proposition}

The value $\sep_i$ is essentially the same as $\sigma_i/(n+1)$, except for the fact that we do not need to assume all elements of $V$ to be used as vertices in $C$. That is, we may have $f_0(C) < |V|$. As the primary example, observe that the link $\lk_C(v)$ of a vertex $v$ of the simplicial complex $C$ can be considered as a complex having as ground set the set of all vertices of $C$ or having only those vertices joined to $v$. The value of the $\sigma$-vector as defined in \cite{Bagchi14StellSpheresTightness} depends on the choice of ground set, but the $\sep$-vector does not.

\begin{theorem}[\protect{\cite[Lemma 4.4]{Novik17NormalizedSigma}}]
\label{thm:ground_set}
Let $C$ be a simplicial complex on a ground set $\Vertices$ but assume that it only uses as vertices a subset $\Vertices_0\subsetneq \Vertices$. Then, $\sep_i(C)$ is independent of whether we compute it using $\Vertices$ or $\Vertices_0$ as a ground set.
\end{theorem}

Another motivation for preferring $\sep$ over $\sigma$ is that it admits a probabilistic interpretation. 
Indeed, the coefficients
\[
P_\Vertices(W):= \frac{1}{(|\Vertices|+1) \binom{|\Vertices|}{|W|}}
\]
add up to 1 when we sum over all subsets $W$ of $\Vertices$, hence they are a probability distribution in $2^\Vertices$. Our $\sep_i(C)$ is nothing but the expected reduced Betti number $\tilde\beta_i$ over all induced subcomplexes of $C$,  with respect to this probability distribution. 

In this interpretation $P_\Vertices$ can be thought as the joint distribution of $|\Vertices|$ binary random variables  (the indicators of the subsets $W\subseteq \Vertices$). Then,~\Cref{thm:ground_set} follows from (and is in fact equivalent to) the fact that for every subset $\Vertices_0\subseteq \Vertices$ of variables we have that $P_{\Vertices_0}$ is the restriction of $P_\Vertices$ to that subset. 

\begin{remark}
\label{rem:probability}
The probability distribution $P_\Vertices$ on $2^{|\Vertices|}$  is equivalent to the \emph{P\'olya urn model with one initial ball of each color} \cite{Mahmoud2008PolyaUrnModel}. Suppose that we have an urn, initially containing one ball of color $0$ and one ball of color $1$. We repeat the following procedure $n$ times: take a ball uniformly at random from the urn, then place it back in the urn together with an additional ball of the same color. It can easily be checked that the probability of obtaining a certain sequence $S\in \{0,1\}^n$ after $n$ trials 
equals $\frac{1}{(n+1) \binom{n}{k}}$ where $k$ is the number of $1$'s in $S$.

One may ask what other probability distributions over $2^{\Vertices}$ can be used to define invariants with similar properties as the $\sep$-vector. There are two nice properties of $P_{\Vertices}$ that are implicitly used throughout this paper, and which a viable alternative probability distribution should satisfy as well:

\begin{enumerate}
\item $P_{\Vertices}$ causes the $|\Vertices|$ individual binary variables to be \emph{exchangeable}. That is, the probability of a subset $W$ only depends on $|W|$ and not on the particular elements it contains. This property makes the $\sep$-vector manageable from the combinatorial point of view, and is also needed in order to have (a statement analogue to)~\Cref{prop:Hochster}.

\item $P_{\Vertices}$ is symmetric under \emph{complementation}. Equivalently, the individual probability of each element is $1/2$. We implicitly use this property often, especially in connection to Alexander duality in~\Cref{sec:DSeq}.
\end{enumerate}
\end{remark}

Assuming both exchangeability and symmetry under complementation, a very natural alternative choice of probability distribution is the uniform distribution, giving probability $(1/2)^{|\Vertices|}$ to every subset. Our reason to prefer $P_{\Vertices}$ is that, this way, the $\sep$-vector allows to determine whether a simplicial manifold is tight or not, see~\Cref{thm:bagchi-intro} and its applications in~\Cref{sec:conjecture}.

\section{$\sep_0$ of (chordal) graphs and strongly connected pure complexes}
\label{sec:chordal}

In this section, we look at the first non-trivial entry of the $\sep$-vector, the 0-dimensional one. By definition, $\sep_i(C)$ only depends on the $(i+1)$-skeleton of $C$, so in particular $\sep_0$ only depends on the $1$-skeleton and can be studied for arbitrary graphs. Computing $\sep_0$ exactly for a given graph is likely to be a computationally hard problem in general, since it is closely related to hard network reliability problems~\cite{Bodlaender04ANote}. Here we show that computing it for chordal graphs is straightforward, and discuss implications for the $\sep_0$ of arbitrary manifolds (or, more generally, for strongly connected pure complexes).
 
\subsection{$\sep_0$ via perfect elimination orderings and in-degree sequences}
\label{ssec:peo}

Let $G=(V,E)$ be a graph on $n=|V|$ vertices, and consider the vertices given in a particular order $v_1,\dots,v_n$. An \emph{in-degree sequence} of $G$ with respect to that ordering is a sequence $(\delta_1,\dots,\delta_n)$ where $\delta_i$ is the number of neighbors of $v_i$ among $\{v_1,\dots,v_{i-1}\}$. Equivalently, it is the in-degree sequence of the digraph obtained from $G$ by directing all edges from the smaller to the larger vertex (according to the given ordering). Let $G_i=G[v_1,\dots,v_i]$ be the subgraph induced by the first $i$ vertices.

An ordering of the vertices of $G$ is called a \emph{perfect elimination ordering} (or {\em p.e.o.}, for short) \cite{Dirac60RigidCircuitGraphs,Fulkerson65ChordalPEO} if for every $i=2,\dots,n$ we have that $G[\{v_j: j<i$ and $v_iv_j\in E\}]$ is a clique in $G$. Equivalently, if $G_{i}$ is obtained from $G_{i-1}$ by joining the new vertex $v_{i}$ to a clique. 

\begin{remark}
\label{rem:chordal}
Dirac's Theorem says that the existence of a p.e.o. in a graph $G$ is equivalent to $G$ being \emph{chordal}; that is, no cycle in $G$ of length greater then three is induced. 

Several generalizations of chordality to higher dimension have been proposed, the closest to our work being the recent homological one by Adiprasito, Nevo and Samper~\cite{ANS16HigherChordality}. More precisely, Adiprasito et al.~call a simplicial complex $C$ \emph{resolution $k$-chordal} if $\sep_k(C)=0$ and \emph{decomposition $k$-chordal} if the $k$-th homology of every induced subcomplex of $C$ is generated by $k$-cycles isomorphic to the boundary of a $(k+1)$-simplex. Among other results, they prove that all complexes with a certain $k$-Dirac property are decomposition $k$-chordal~\cite[Proposition 6.3]{ANS16HigherChordality} and explore the converse implication.
\end{remark}

\begin{lemma}
\label{lemma:peo}
Let $G$ be a graph with in-degree sequence $(0=\delta_1,\dots,\delta_n)$ and, as above, let $G_i$ denote the  subgraph induced by the first $i$ vertices. Then,
\begin{equation}
\label{eq:peo} 
\sep_0(G_i) - \sep_0(G_{i-1}) \le \frac{1}{(\delta_i+1)(\delta_i+2)} - \frac{1}{i(i+1)} , 
\end{equation}
with equality if and only if $G[\{v_j: j<i$ and $v_iv_j\in E\}]$ is a clique.
In particular, 
\[
\sep_0(G) \le \frac1{n+1} + \sum_{i=1}^n \frac{1}{(\delta_i+1)(\delta_i+2)} - 1,
\]
with equality if and only if the ordering $v_1,\dots,v_n$ is a perfect elimination ordering.
\end{lemma}

\begin{proof}
The second part of the lemma easily follows from the first one by induction on $i$ (for the base case $i=1$ observe that $\delta_1=0$).

For the first part, denote $V_i=\{v_1,\dots,v_i\}$ for each $i$ and rewrite $\sep_0(G_{i})$ as
 \begin{align*}
{\sep}_0(G_i) &
=\frac1{i+1}\sum_{W\subseteq V_{i-1} } \frac{\tilde{\beta_0}(G_i[W])}{\binom{i}{|W|}} 
+\frac1{i+1}\sum_{W\subseteq V_{i-1} } \frac{\tilde{\beta_0}(G_i[W\cup\{i\}])}{\binom{i}{|W|+1}}. 
\\
 \end{align*}

By \Cref{thm:ground_set}, we can consider $G_{i-1}$ as a graph on the ground set $V_i$ and write
 \begin{align*}
{\sep}_0(G_{i-1}) &
=\frac1{i+1}\sum_{W\subseteq V_{i-1} } \frac{\tilde{\beta_0}(G_i[W])}{\binom{i}{|W|}} 
+\frac1{i+1}\sum_{W\subseteq V_{i-1} } \frac{\tilde{\beta_0}(G_i[W])}{\binom{i}{|W|+1}}. \\
 \end{align*}

The difference of these two expressions gives
 \begin{align*}
{\sep}_0(G_{i}) - {\sep}_0(G_{i-1}) &
=\frac1{i+1}\sum_{W\subseteq V_{i-1} } \frac{\tilde{\beta_0}(G_i[W\cup\{i\}]) - \tilde{\beta_0}(G_i[W])}{\binom{i}{|W|+1}}.
 \end{align*}

We thus need to look at the difference $\tilde{\beta_0}(G_i[W\cup\{i\}]) - \tilde{\beta_0}(G_i[W])$. 
For each $W\subseteq V_{i-1}$ we distinguish depending on whether $W$ contains some neighbor of $v_i$ or not:

\begin{itemize}
\item If $W$ contains no neighbor of $v_i$ then $\tilde{\beta_0}(G_i[W\cup\{i\}]) - \tilde{\beta_0}(G_i[W]) =1$ except for the case $W=\emptyset$, where it equals zero. The contribution of all such subsets $W$ to ${\sep}_0(G_{i}) - {\sep}_0(G_{i-1})$ can thus be computed exactly.
\begin{align*}
\label{eq:peo1}
\frac{1}{i+1} \sum_{k=1}^{i-1-\delta_i} \frac{ \binom{i-1-\delta_i}{k} } {\binom{i}{k+1}}
= \frac{1}{i+1} \sum_{k=0}^{i-1-\delta_i} \frac{ \binom{i-1-\delta_i}{k} } {\binom{i}{k+1}} - \frac1{i(i+1)} 
=\frac1{(\delta_i+1)(\delta_i+2)} - \frac1{i(i+1)}.
 \end{align*}
The last equality is a consequence of \Cref{lemma:binomial_quotients} below.

\item If $W$ contains neighbors of $v_i$ then $\tilde{\beta_0}(G_i[W\cup\{i\}]) - \tilde{\beta_0}(G_i[W])\le 0$, and equality holds for every  $W$ if and only if the neighbors form a clique. This finishes the proof.
\end{itemize}
\end{proof}

The following combinatorial identity used in the proof of \Cref{lemma:peo} appears several times in the paper.

\begin{lemma}
\label{lemma:binomial_quotients}
For any non-negative integers  $n$ and $a\le b$ one has
\[
\sum_{k=0}^{n}\frac{\binom{n}{k}}{\binom{n+b}{k+a}}
= \frac{n+b+1}{(b+1)\binom{b}a}.
\]
\end{lemma}

\begin{proof}
\begin{align*}
\label{eq:binomial_quotients}
\binom{b}{a}\sum_{k=0}^{n}\frac{\binom{n}{k}}{\binom{n+b}{k+a}}
   &= \frac{b!}{a!(b-a)!}\sum_{k=0}^{n} \frac{n!(n+b-k-a)!(k+a)!}{(n+b)!(n-k)!k!} \\
   &= \frac{n!b!}{(n+b)!}  \sum_{k=0}^{n} \frac{(n+b-k-a)!(k+a)!}{(n-k)!k!a!(b-a)!} \\
   &= \frac{n!b!}{(n+b)!}  \sum_{k=0}^{n} \binom{n+b-k-a}{b-a} \binom{k+a}a \\
   &= \frac{n!b!}{(n+b)!}  \binom{n+b+1}{b+1}
   = \frac{n+b+1}{b+1}.
\end{align*}

The second to last equality 
\[
 \sum_{k=0}^{n} \binom{n+b-k-a}{b-a} \binom{k+a}a = 
 \binom{n+b+1}{b+1}
\]
is a form of the Chu-Vandermonde identity, and follows from the fact that the left-hand side enumerates subsets of $[n+b+1]$ of size $b+1$: each summand counts the possibilities for the $(a+1)$-th element in the subset to be  $k+a+1$.
\end{proof}

\subsection{Billera-Lee spheres and balls}
\label{ssec:BLSpheres}

 The Billera-Lee spheres are simplicial polytopal $(d-1)$-spheres realizing all possible $f$-vectors allowed by McMullen's conditions, see \cite{Billera81ProofSuffMcMullenCondFVec}. They are constructed as the boundaries of certain $d$-balls that we call \emph{Billera-Lee balls}. We here introduce this construction, motivated by the fact (proven in \Cref{ssec:upperbound}) that Billera-Lee $d$-balls maximize $\sep_0$ among all strongly connected $d$-complexes with a fixed number of vertices and edges.

Consider the cyclic $(d+1)$-polytope with $n$ vertices, $C_{d+1}(n)$. In its standard embedding along the moment curve, we define its \emph{lower facets} to be the facets whose exterior normal has a negative last coordinate. By Gale's evenness criterion (see, for instance, \cite[Corollary 6.1.9]{DeLoera10Triangulations}), the lower facets of $C_{d+1}(n)$ are

\begin{itemize}
  \item $(i_1, i_1+1, \ldots , i_k, i_k + 1)$, for $1 \le i_1 < \dots < i_k \le n-1$ non-consecutive, if $d+1=2k$ is even; and
  \item $(1, i_1, i_1+1, \ldots , i_k, i_k + 1)$, for $2 \le i_1 < \dots < i_k \le n-1$ non-consecutive, if $d+1=2k+1$ is odd.
\end{itemize}

\begin{figure}[hbt]
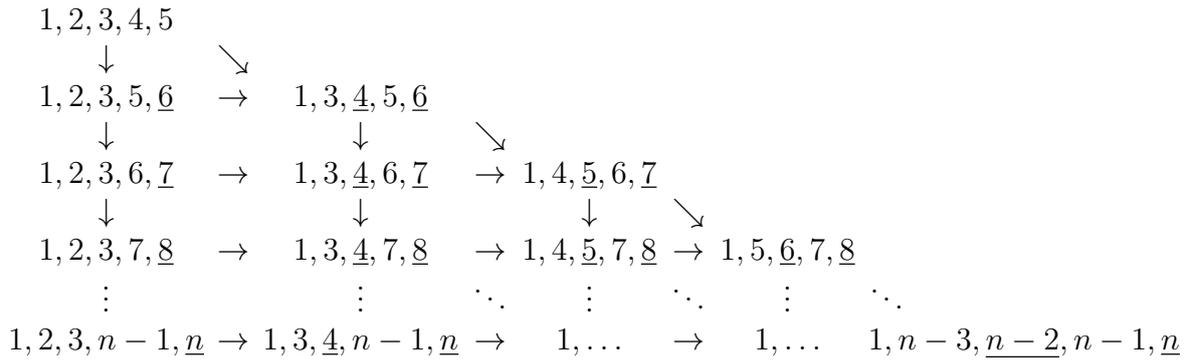

\setlength{\arraycolsep}{2.5pt}
\[
\begin{array}{cccccccl}
1,2,3,4,5 \\
\downarrow & \searrow \\
1,2,3,5,\uu6 &\to& 1,3,\uu4,5,\uu6 \\
\downarrow && \downarrow & \searrow \\
1,2,3,6,\uu7 &\to& 1,3,\uu4,6,\uu7 &\to& 1,4,\uu5,6,\uu7 \\
\downarrow && \downarrow && \downarrow & \searrow \\
1,2,3,7,\uu8 &\to& 1,3,\uu4,7,\uu8 &\to& 1,4,\uu5,7,\uu8 &\to& 1,5,\uu6,7,\uu8 \\
\vdots && \vdots &\ddots& \vdots &\ddots& \vdots & \ddots \\
1,2,3,n-1,\uu{n} &\to& 1,3,\uu4,n-1,\uu{n} &\to& 1,\dots &\to& 1,\dots & 1,n-3,\uu{n-2},n-1,\uu{n} \\
\end{array}
\]
  \caption{Directed dual graph, or shelling order, of the $4$-ball that is the lower envelope of the cyclic $5$-polytope with $n$ vertices. \label{fig:lowerEnv}}
\end{figure}

As a complex, these lower facets form a $d$-ball that we denote $L_{d}(n)$. See \Cref{fig:lowerEnv} for the list of facets of $L_{4}(n)$. As shown in the figure, we consider the adjacency graph of $L_{d}(n)$ with its edges directed towards the lexicographically larger facet incident to the ridge they represent. This orientation is obviously acyclic. It corresponds to a shelling of $L_{d}(n)$ in the sense that any linear ordering compatible with this orientation is a shelling order. 
Since the in-degree of each facet in this directed graph equals the number of ridges it has in common with the complex it is glued to in the shelling process, the $h$-vector of $L_{d}(n)$ has as $h_i$ the number of facets of in-degree $i$. An easy calculation shows that
\[
h_i = \binom{n-d-2+i}{i}, \qquad \forall i = 0,\dots, \lceil d/2\rceil,
\]
and $h_i=0$ for $i = \lceil d/2\rceil + 1, \dots, d+1$.

Let $B$ be any ideal in the partial order. That is, let $B$ be a subset of facets of $L_{d}(n)$ such that $F\in B$ and $F'\to F$ is a directed edge implies $F'\in B$. Then $B$ is an initial segment of a shelling of $L_{d}(n)$ and, in particular, $B$ is a shellable $d$-ball whose $h_i$ equals the number of facets $F$ in $B$ that have in-degree equal to $i$.
The {\em Billera-Lee balls} are given by some particular ideals of this type.

\begin{theorem}[Billera and Lee~\cite{Billera81ProofSuffMcMullenCondFVec}]
Let $k=(k_0,\dots, k_{\lfloor d/2\rfloor})$ be any vector satisfying McMullen's conditions for the $g$-vector of a $d$-polytope and let $n\ge k_0+d+1$.
Let $B^d_k$ be the subset of facets of $L_{d}(n)$ given by 
\[
B^d_k:= \bigcup_{i=0}^{\lfloor d/2\rfloor} \{ \text{first $k_i$ facets of $L_{d}(n)$ with in-degree $i$, in reverse lexicographical order} \}.
\]
Then, $B^d_k$ is an ideal in the partial order of facets of $L_{d}(n)$. In particular, it is a shellable $d$-ball with $h$-vector equal to $k$. Moreover, $\partial B^d_k$ is polytopal.
\end{theorem}

We call the ball $B^d_k$ constructed in the theorem the \emph{Billera-Lee ball}  with $h$-vector equal to the given $k$, and its boundary $S^{d-1}_k : = \partial B^d_k$ the \emph{Billera-Lee sphere} with $g$-vector 
\[
(k_0,\dots, k_{\lfloor d/2\rfloor-1}, k_{\lfloor d/2\rfloor}, -k_{\lceil d/2\rceil}, - k_{\lceil d/2\rceil-1}, \dots, -1).
\]
Here $k_{\lceil d/2\rceil}$ is considered to be zero when $d$ is odd. That this is indeed the $g$-vector of $S^{d-1}_k$ follows from the fact that the $h$-vector of a ball and the $g$-vector of its boundary are related by $g(\partial B) = h(B) - \overline h(B)$, where $\overline h(B)$ denotes the $h$-vector written in reverse. 
(See \cite[Theorem 2.6.11]{DeLoera10Triangulations} or \cite[Chapter 17, Theorem 7.3.6]{HandbookDCG}.)

\begin{example}
  For $d=4$, assume that we want to construct a $3$-sphere with prescribed $g$-vector $(g_1,g_2)$. That is, in the above notation we want to construct $S^3_{(1,g_1,g_2)} = \partial B^4_{(1,g_1,g_2)}$. In \Cref{fig:lowerEnv}, first take facet $12345$. This is the only facet with in-degree $0$ and hence the only facet that contributes to $h_0$ of $B^4_{(1,g_1,g_2)}$. Then take the next $g_1$ facets of the first column (they contribute to $h_1$), and the first $g_2$ facets of the other columns in the order they are read ($13456, 13467, 14567, 13478, \ldots$, they contribute to $h_2$). By construction, the boundary of this ball is then a $3$-sphere with $g$-vector $(g_1,g_2)$. 
  
Observe that each row, considered as a sequence of flips in the boundary of the previously constructed ball $B^4_{(1,g_1,g_2)}$, connects the vertex $i$, inserted by a $(1,4)$-flip by the first facet in the row, to all other vertices of the sphere. In particular, the last row corresponds to a sequence of flips turning $C_{4}(n-1)$ into $C_{4}(n)$. In \Cref{fig:lowerEnv}, the (minimal) faces introduced by the respective flip are underlined.
\end{example}

For $d \ge 3$, number of vertices $f_0\ge d+1$ and number of edges $f_1\in  [d f_0 - \binom{d+1}2), \binom{f_0}{2}]$, the graph of the Billera-Lee $d$-ball with $f_0$ vertices and $f_1$ edges consists of {\em a)} a clique of size $k \in [d+1, f_0]$, {\em b)} if $k<f_0$, 
a vertex attached to $j \in [d, k-1]$ vertices from that clique, and {\em c)} $f_0-k-1$ additional vertices each attached to $d$ vertices forming a $d$-clique in the previous list. The parameters $j$ and $k$ can be deduced from $f_0$, $f_1$ and $d$ in the following way: $k$ is the largest integer such that $\binom{k}2 + (f_0-k)d \le f_1$, and $j= f_1 - \binom{k}2 - (f_0-k-1)d$.

\begin{definition}
We call a graph obtained in this way a {\em Billera-Lee graph} with parameters $(f_0,f_1,d)$, and denote it by $G(f_0,f_1,d)$. 
\end{definition}

Note that, for dimension $d\ge 3$, we have that $f_1(B^d_k)  = f_1(S^d_k)$ as a direct consequence of the construction of $B^d_k$. Hence, we can observe that the graph of the Billera-Lee $(d-1)$-sphere with $f_0$ vertices and $f_1$ edges is isomorphic to some Billera-Lee graph with parameters $(f_0,f_1,d)$.

Since the given ordering of vertices is a p.e.o., we can deduce $\sep_0$ of a Billera-Lee graph from \Cref{lemma:peo}.

\begin{corollary}
With the above notation, the $\sep_0$-value of a Billera-Lee graph $G$ with parameters $(f_0,f_1,d)$ equals
\[
\sep_0(G) = \frac1{f_0+1} - \frac1{k+1}  + \frac1{(j+1)(j+2)} + \frac{f_0-k-1}{(d+1)(d+2)}.
\]

\end{corollary}

\subsection{In-degree sequences of pure complexes and an upper bound for $\sep_0$}
\label{ssec:upperbound}

We call an in-degree-sequence \emph{$d$-dimensional} if it is of the form $(0,1,\dots,d, k_1,\dots,k_{n-d-1})$ with $k_i \geq d$ for all $1 \leq i \leq n-d-1$. 
Note that every $d$-dimensional in-degree sequence is also $(d-1)$-dimensional. The $d$ only indicates a lower bound for the entries in the sequence.

\begin{proposition}
  \label{prop:indegree}
  Every pure $d$-complex with connected adjacency graph (that is, every {\em strongly connected} pure simplicial complex) has an ordering of its vertices with a $d$-dimensional in-degree sequence.
\end{proposition}

\begin{proof}
 In order to construct a $d$-dimensional in-degree sequence, build up the complex step by step adding one vertex at the time and forming the respective induced subcomplex. Start with the vertices of a facet to obtain the first $d+1$ entries of the sequence. By the connectedness assumption, in any further induced subcomplex we must have some facet $F$ for which some adjacent facet $F'$ is still not in the subcomplex. Since the subcomplex is induced, the unique vertex of $F'\setminus F$ is not in it either. Since $F'$ and $F$ are adjacent, this vertex must connect to at least $d$ vertices of the current subcomplex. Add this vertex to construct the next induced subcomplex.
\end{proof}

\begin{lemma}
  \label{lem:billeralee}
  Let $C$ be a graph with a $d$-dimensional in-degree sequence  and let $G$ be the Billera-Lee graph with parameters $(f_0(C),f_1(C),d)$. Then 
  $\sep_0 (C) \leq \sep_0 (G)$.
\end{lemma}

\begin{proof}
  Let $\delta = (0,1,2,3,\dots,d, k_1,\dots,k_{n-d-1})$ be a $d$-dimensional in-degree sequence realizing $C$. Such a sequence exists due to \Cref{prop:indegree}. Then $\sep_0(C)$ satisfies the upper bound specified in \Cref{lemma:peo}. The proof is completed by the observation that this upper bound 
increases when $\delta$ is modified into a $d$-dimensional in-degree sequence of Billera-Lee type keeping the sum of degrees constant.
\end{proof}

\begin{corollary}
  \label{cor:billeralee}
  Let $C$ be a strongly connected $d$-complex, and let $G$ be the Billera-Lee graph with parameters $(f_0(C),f_1(C),d)$. Then 
  $\sep_0 (C) \leq \sep_0 (G)$.
\end{corollary}

Combining all these observations we have the following application of \Cref{cor:billeralee}.

\begin{theorem}
  \label{thm:upperbound}
  Given $(f_0,f_1)$ and $d\ge 3$, the maximum $\sep_0$ among all strongly connected $d$-complexes with $f_0$ vertices and $f_1$ edges lies between the $\sep_0$ of the Billera-Lee graph with parameters $(f_0,f_1,d+1)$ (realized by a Billera-Lee $d$-sphere) and that of the Billera-Lee graph with parameters $(f_0,f_1,d)$. The quotient between these two values is smaller than $(d+3)/(d+1)$.
\end{theorem}

\begin{proof}
  The lower bound follows from the existence of the Billera-Lee $d$-spheres (which have as graph the Billera-Lee graph with parameter $(f_0,f_1,d+1)$). The upper bound is implied by \Cref{cor:billeralee}. 
\end{proof}

With the help of the results in this section we can also give an elementary proof of the following special case of \Cref{conj:general-intro}.

\begin{proposition}
  \label{thm:nearlyneighb}
  Let $G$ be an $n$-vertex graph with at most $n-d-2$ missing edges (that is, 
  $f_1 \geq {n \choose 2} - n + d + 2$). Then, $G$ admits a $(d+1)$-dimensional in-degree sequence. In particular,
   \[ \sep_0 (G) \leq \sep_0(\sphere), \]
     where $\sphere$ is an $n$-vertex Billera-Lee $d$-sphere with $f_1$ edges.
\end{proposition}

\begin{proof}
  The second part of the statement follows from the first part by \Cref{lem:billeralee}.
  
  To prove the first part, let us first show that $G$ contains a clique of size $d+2$. For this, let $K$ be any maximal clique in $G$. Maximality means $G$ has at least $n-|K|$ missing edges, so that $n-|K| \le n-d-2$ and $|K|\ge d+2$.

We now construct the vertex-sequence starting with a clique of size $d+2$ and then, for every  $\ell \in d+2,\dots,n-1$, we choose as $\ell+1$-th vertex any one with maximum number of neighbors among the already chosen ones.
We claim that this number of neighbors is always at least $d+1$. 

If it were not,
each of the $n-\ell$ choices of the next vertex has at most $d$ edges connecting it to the first $\ell$. This implies 
 at least $(\ell-d )(n-\ell)$ edges missing from $G$. 
  Since $(\ell-d)(n-\ell) > n-(d+2)$ for all $d+1 \leq \ell \leq n-1$, we have a contradiction.
\end{proof}

Observe that the first part of the statement is not true if $n-d-1$ edges are missing. For example, the graph of a $(d+1)$-cross-polytope has no in-degree sequence starting with $(0,1,2,\dots,d+1,d+1)$.

\section{The $\sep$-vector under manifold operations}

This section is dedicated to a detailed analysis of how the $\sep$-vector behaves under certain standard operations on simplicial manifolds.

\subsection{Bistellar flips}
\label{ssec:bistflips}

\emph{Bistellar flips}, \emph{bistellar moves}, or just \emph{flips} are local modifications that change one triangulation of a manifold into another without changing the PL-topological type of the manifold. 

They are sometimes called \emph{Pachner moves} due to the following result of Pachner \cite{Pachner87KonstrMethKombHomeo, Pachner91PLHomeo}:  two triangulated manifolds $C_1$ and $C_2$ are PL-homeomorphic if and only if there is a sequence of flips transforming one in the other.

For the precise definition we need the following setup: let $F$ be a ground set of size $d+2$, and for each $ \emptyset \ne F_1\subsetneq F$, let $B_{F,F_1}$ be defined as the simplicial complex with vertex set $F$ and unique minimal non-face $F_1$. More explicitly, $B_{F,F_1}$ is the pure $d$-complex with $|F_1|$ facets, given by the subsets $\{F \setminus \{i\} : i \in F_1\}$.

It is a fact that $B_{F,F_1}$ is always a simplicial ball and that its boundary equals the boundary of $B_{F, F_2}$, where $ F_2 = F \setminus F_1$. It is also obvious that the isomorphism type of $B_{F,F_i}$ only depends on $|F|$ and $|F_i|$.

\begin{definition}[Bistellar flip]
Let $\manifold_1$ be a triangulated $d$-manifold on a ground set $V$ and let $F\subseteq V$ be of size $d+2$. Suppose that $\manifold_1[F]= B_{F,F_1}$ for some $\emptyset \ne F_1\subsetneq F$ and let $F_2=F\setminus F_1$. Assume $F_2$ is not a face in $\manifold_1$. 
Then, the \emph{bistellar flip} of $F_1$ (or of $F$) in $\manifold_1$ is the triangulated manifold $\manifold_2$ obtained by removing the subcomplex $B_{F,F_1}$ from $\manifold_1$ and inserting $B_{F,F_2}$ in its place. 
We say that the flip is \emph{of type} $(|F_1|, |F_2|)$ and we call $F_1$ and $ F_2$ the \emph{face removed} and \emph{face inserted} by the flip, respectively. (More precisely, the flip inserts/removes all faces containing $F_1$ and $F_2$, respectively; $F_1$ and $F_2$ are the unique minimal inserted/removed faces).
\end{definition}

A flip of type $(i,j)$, $i+j=d+2$, replaces $i$ facets in a triangulation by $j$ of them. Flips of types $(i,j)$ and $(j,i)$ are inverse operations to one another. 
More precisely, for all choices of $F$ and $F_1$, the simplicial complex $B_{F,F_1} \cup B_{F, F_2}$ is isomorphic to the boundary of the $(d+1)$-simplex.

Flips of type $(1,d+1)$ are also called \emph{stellar subdivisions at a facet} or \emph{stacking operations}. Spheres obtained from the boundary of a simplex by stacking operations are called \emph{stacked} (see \Cref{ssec:stacked}). Note that, when performing flips of type $(1,d+1)$, a new element of the ground set that is not a vertex of $\manifold_1$ has to be added to the set of vertices of $\manifold_2$. When talking about $\sep$ this is not an issue, since it is independent of the ground set in use, see \Cref{thm:ground_set}.

\begin{remark}
\label{rem:flip_facevectors}
It is straightforward to describe how a bistellar flip changes the $f$-, $h$-, and $g$-vector of a $d$-manifold. For the latter, this takes the following very simple form.
If $\manifold_2$ is obtained from $\manifold_1$ by a flip of type $(i,j)$  then
\[
g_i(\manifold_2) = g_i(\manifold_1) +1, \quad
g_j(\manifold_2) = g_j(\manifold_1) -1,  \quad
g_k(\manifold_2) = g_k(\manifold_1) \text{ for all $k\not\in \{i,j\}$}.
\]
(Assuming $i\ne j$. Flips with $i=j$ do not change the $f$-, $h$-, or $g$-vector).
\end{remark}

The homotopy type of most induced subcomplexes are not affected by a bistellar flip, which allows us to give a qualitative statement on how the $\sep$-vector of the manifold changes under the operation.

\begin{lemma}[Lemma 2.3, parts 1 and 2a in \cite{Bha2016}]
\label{lemma:flip_subcomplexes}
Let $\manifold_1$ and $\manifold_2$ be $d$-manifolds, $\manifold_2$ obtained from $\manifold_1$ by a flip with removed face $F_1$ and inserted face $F_2$. 
Let $F=F_1\cup F_2$, $i=|F_1|$ and $j=|F_2|$. In particular, the flip is of type $(i,j)$.
Let  $W\subseteq V$ be a subset of the ground set. 
\begin{itemize}
\item If $W \cap F \notin\{F_1,F_2\}$ then $\manifold_1[W] \simeq \manifold_2[W]$.
\item If $W \cap F = F_1$ then $\tilde \beta_k(\manifold_2[W]) = \tilde \beta_k(\manifold_1[W]), \quad \forall k\not\in \{i-1,i-2,j-1,j-2\}$.
If, moreover,  $i\ne j$ then 
\[
\begin{array}{rclrcl}
\tilde \beta_{i-2}(\manifold_2[W])  &\le& \tilde \beta_{i-2}(\manifold_1[W]), &\qquad
\tilde \beta_{j-2}(\manifold_2[W])  &\ge& \tilde \beta_{j-2}(\manifold_1[W]), \\
\tilde \beta_{i-1}(\manifold_2[W])  &\ge& \tilde \beta_{i-1}(\manifold_1[W]), &
\tilde \beta_{j-1}(\manifold_2[W])  &\le& \tilde \beta_{j-1}(\manifold_1[W]). \\
\end{array}
\]
\item If $W \cap F = F_2$ then the same happens, with inequalities in the opposite direction.
\end{itemize}
\end{lemma}

\begin{proof}
For the first part, just observe that if $W \cap F \notin\{F_1,F_2\}$ then $\manifold_1[W\cap F]$ and $\manifold_2[W\cap F]$ deformation retract to $\manifold_1[W\cap F]\cap \manifold_2[W\cap F]$. 

For the second part, assume $F=F_1$ (the case $F=F_2$ is similar, since it is the reverse flip). 
Removing the $(i-1)$-face $F_1$ can only decrease $\tilde \beta_{i-1}$ and increase $\tilde \beta_{i-2}$, and inserting the $(j-1)$-face $F_2$ can only increase $\tilde \beta_{j-1}$ and decrease $\tilde \beta_{j-2}$.
\end{proof}

\begin{corollary}[Lemma 2.3, parts 1 and 2a in \cite{Bha2016}]
\label{coro:flip_sepind}
Let $\manifold_1$ and $\manifold_2$ be two $d$-manifolds, with $\manifold_2$ obtained from $\manifold_1$ by an $(i,j)$-flip. Then $\sep_k(\manifold_2) = \sep_k(\manifold_1), \, \forall k\not\in \{i-2,i-1,j-2,j-1\}$.
If, moreover,  $i\ne j$ then
\[
\begin{array}{rclrcl}
\sep_{i-2}(\manifold_2)  &<& \sep_{i-2}(\manifold_1), &\qquad
\sep_{j-2}(\manifold_2)  &>& \sep_{j-2}(\manifold_1), \\
\sep_{i-1}(\manifold_2)  &>& \sep_{i-1}(\manifold_1), &
\sep_{j-1}(\manifold_2)  &<& \sep_{j-1}(\manifold_1). \\
\end{array}
\]
\end{corollary}

\begin{proof}
All statements, except for the strictness of the inequalities, follow directly from \Cref{lemma:flip_subcomplexes}. Strictness follows from considering the cases $W=F_1$ and $W=F_2$.
\end{proof}

\subsection{Connected sum}
\label{ssec:connsum}

Building simplicial connected sums or, conversely, decomposing a  manifold into its connected summands, is sometimes applied as a step to organise proofs and/or to provide an argument with additional combinatorial structure. For instance, decomposing a triangulated $2$-sphere along all of its induced $3$-cycles yields a collection of flag $2$-spheres (plus possibly some boundaries of the tetrahedron). More generally, the $1$-skeleton of a $d$-manifold, which is always at least $(d+1)$-connected, is at least $(d+2)$-connected if and only if it is not a simplicial connected sum.

\begin{definition}[Simplicial connected sum]
  Given two triangulated $d$-manifolds $\manifold_1$ and $\manifold_2$, their {\em simplicial connected sum}, written $\manifold_1 \# \manifold_2$ is obtained by the following procedure. Remove a facet from each $\manifold_1$ and $\manifold_2$, and glue the resulting boundaries (both isomorphic to $\partial \Delta_{d}$).
\end{definition}

The combinatorics of $C_1\# C_2$ depends on the choice of $F_1$ and $F_2$,  but its topology only depends on (the parity of) the bijection between $F_1$ and $F_2$ (in case both $C_1$ and $C_2$ are chiral, i.e., orientable and do not admit orientation-reversing automorphisms). Our next result implies that the $\sep$-vector does not depend on this choice either.

\begin{theorem}
\label{thm:connectedsum}
Let $\manifold_1, \manifold_2$ be $d$-manifolds for $d\ge 2$. Let $n_i$ be the number of vertices of $\manifold_i$. Then
\begin{align*}
 {\sep}_i(\manifold_1 \# \manifold_2) &=  \sep_i(\manifold_1) + \sep_i(\manifold_2) +  
 c(d,n_1,n_2),  & i \in \{0,d-1\},
 \\
 {\sep}_i(\manifold_1 \# \manifold_2) &= \sep_i(\manifold_1) + \sep_i(\manifold_2), & i\in \{1,\dots,d-2\}.
\end{align*}
Here, 
\begin{align*}
 c(d,n_1,n_2)
&:= \frac{1}{d+2} - \frac{1}{n_1 +1}- \frac{1}{n_2 +1} +\frac{1}{n_1+n_2-d-1}\\
&= \frac{1}{d+2} - \frac{1}{n_1 +1}- \frac{1}{n_2 +1} +\frac{1}{n+1},
\end{align*}
where  $n = n_1+n_2 - d- 1$ denotes the number of vertices of $\manifold_1 \# \manifold_2$.
\end{theorem}

\begin{proof}
Let $\manifold= \manifold_1 \# \manifold_2$. Denote by $V_k$ the set of vertices of $\manifold_k$, $k=1,2$, 
and let $D = V_1 \cap V_2$ be the common facet along which we take the connected sum.
Let $\tilde\manifold_k= \manifold_k \setminus \{D\}$, so that $\manifold_1 \# \manifold_2 = \tilde\manifold_1 \cup \tilde\manifold_2$.
By \Cref{thm:ground_set} we regard all these  complexes as having the same ground set $V=V_1 \cup V_2$. That is,
\begin{align*}
\sep_i(M)= & \sum_{W\subseteq V} \frac{\tilde \beta_i(M[W])}{(|V|+1) \binom{|V|}{|W|}}, \quad
\sep_i(M_k)=  \sum_{W\subseteq V} \frac{\tilde \beta_i(M_k[W])}{(|V|+1) \binom{|V|}{|W|}}.
\end{align*}

For $i\in \{1,\dots,d-2\}$ we have that
\begin{align}
\label{eq:connsumproof}
\beta_i(\manifold[W]) 
= \beta_i(\tilde \manifold_1[W]) + \beta_i(\tilde\manifold_2[W])
= \beta_i( \manifold_1[W]) + \beta_i(\manifold_2[W])
\end{align}
for every $W\subseteq V$. The first equality comes from Mayer-Vietoris, thanks to the fact that 
$\tilde \manifold_1[W] \cap \tilde\manifold_2[W]$ is either empty, contractible, or a $(d-1)$-sphere. (In particular, its $i$-th and $(i-1)$-th Betti numbers vanish). The second equality comes from the fact that removing a $d$-face only affects homology in dimensions $d$ and $d-1$. This implies 
$ {\sep}_i(\manifold_1 \# \manifold_2) = \sep_i(\manifold_1) + \sep_i(\manifold_2)$ for $ i\in \{1,\dots,d-2\}$.

For $\tilde \beta_0$ the second equality in~\Cref{eq:connsumproof} still holds but the first one needs a correction term whenever $W\cap D=\emptyset$ but both $W\cap V_k \ne\emptyset$.
More precisely, denoting $W_D=W\cap D$, $W_1=W\setminus V_2$ and $W_2=W\setminus V_1$, we have
\begin{align}
\label{eq:connsum0}
\tilde{\beta_0}(\manifold[W]) = \tilde{\beta_0}(\manifold_1[W]) + \tilde{\beta_0}(\manifold_2[W]) + \left\{\begin{aligned}
& 0 && \text{if }W_D\neq \emptyset, W_1 = \emptyset \textrm{ or } W_2  = \emptyset, \\
& 1 && \text{if } W_D=\emptyset \textrm{ and } W_1 \neq \emptyset \neq W_2.
\end{aligned}
\right.
\end{align}

For $i=d-1$ we have the opposite. If $D\not\subseteq W$ (that is, if $W_D \ne D$)~\Cref{eq:connsumproof} still holds, and the same happens if $V_1\subseteq W$ or $V_2\subseteq W$  (in this last case both inequalities in~\eqref{eq:connsumproof} may fail, but their failures cancel out  since the $(d-1)$-cycle $\partial D$ is trivial in all of $M_1[W]$, $M_2[W]$ and $M[W]$).
However, if $D\subseteq W$ but $W$ contains none of $V_1$ or $V_2$ then $\partial D$ adds one to $\beta_{d-1}(\manifold[W])$. That is,
\begin{align}
\label{eq:connsumd}
\beta_{d-1}(\manifold[W]) = \beta_{d-1}(\manifold_1[W]) + \beta_{d-1}(\manifold_2[W]) + \left\{\begin{aligned}
& 0 && \text{if }W_D\neq D, W_1 =V_1,  \textrm{ or } W_2  = V_2; \\
& 1 && \text{if } W_D=D \textrm{ and } W_1 \neq \emptyset \neq W_2.
\end{aligned}
\right.
\end{align}

Observe that a set $W$ contributes to the correction term in~\Cref{eq:connsum0} if and only if  its complement contributes in~\Cref{eq:connsumd}. Thus, the global correction is identical and, due to \Cref{eq:connsum0}, equal to

\begin{align*}
&\qquad \qquad\quad 
\sum_{\substack{\emptyset \neq W_1\subseteq V_1 \setminus D,\\ \emptyset \neq W_2\subseteq V_2 \setminus D}} \frac{1}{(n+1)\binom{n}{|W_1 \cup W_2|}} =
\\
= & \frac{1}{n+1} \sum_{i=1}^{n_1-d-1} \sum_{j=1}^{n_2-d-1} \frac{\binom{n_1 -d -1}{i}\binom{n_2-d-1}{j}}{\binom{n}{i+j}}  
\\
= & \frac{1}{n+1} \left( \sum_{i=0}^{n_1-d-1} \sum_{j=0}^{n_2-d-1} \frac{\binom{n_1 -d -1}{i}\binom{n_2-d-1}{j}}{\binom{n}{i+j}} - 
\sum_{j=0}^{n_2-d-1} \frac{\binom{n_2-d-1}{j}}{\binom{n}{j}} - \sum_{i=0}^{n_1-d-1} \frac{\binom{n_1-d-1}{i}}{\binom{n}{i}} +1 \right) 
\\
= & \frac{1}{n+1} \left( \sum_{k=0}^{n-d-1} \frac{\binom{n-d-1}{k}}{\binom{n}{k}} - \sum_{j=0}^{n_2-d-1} \frac{\binom{n_2-d-1}{j}}{\binom{n}{j}} - \sum_{i=0}^{n_1-d-1} \frac{\binom{n_1-d-1}{i}}{\binom{n}{i}} +1 \right) 
\\
 = & \frac{1}{d+2} - \frac{1}{n_1 +1}- \frac{1}{n_2 +1} +\frac{1}{n+1}.
\end{align*}

Here the last equality is implied by a special case of \Cref{lemma:binomial_quotients}, namely
\[
\sum_{k=0}^{n-t}\frac{\binom{n-t}{k}}{\binom{n}{k}}
= \frac{n+1}{t+1}.
\]
\end{proof}

Observe that, as a consequence of \Cref{thm:connectedsum}, we have that if $M_1$ and $M_2$ are triangulated $d$-manifolds with their $\sep$-vectors bounded by those of the Billera-Lee spheres with their respective $f$-vectors, then $M_1 \# M_2$ also has a $\sep$-vector bounded by that of the corresponding Billera-Lee sphere.

As another application of \Cref{thm:connectedsum}, we compute the change of the $\sep$-vector under arbitrarily many stacking operations (i.e. bistellar $(1,d+1)$-moves). Note that stacking is the same as performing a simplicial connected sum of a complex with the boundary of a $(d+1)$-simplex, which has $d+2$ vertices and whose $\sep$-vector vanishes except for $\sep_{-1}=\sep_d=1$. 
The graded Betti numbers of stacked spheres (spheres obtained from a simplex by stacking operations) were computed in \cite{TeraiHibi97}.

\begin{corollary}[see also Lemma 2.3, part 2b in \cite{Bha2016}]
 \label{cor:stacking}
 Let $\manifold$ be an $n$-vertex combinatorial $d$-manifold, $d\geq 2$. Then for every combinatorial $d$-manifold $\altmanifold$ obtained from $\manifold$ by $k$ bistellar $(1,d+1)$-moves (or stacking operations) we have
\begin{align}
 \label{eq:stacking}
 \sep_j (\altmanifold) &=  \sep_j (\manifold) + \frac{k}{(d+2)(d+3)} - \frac{1}{n+1} +\frac{1}{n+k+1}, & j \in \{ 0, d-1 \}, \\
 \sep_i (\altmanifold) &= \sep_i (\manifold) 
 , & i \in \{1,\dots,d-2 \}. \nonumber 
 \end{align}
\end{corollary}

\begin{proof}
We prove the result for $\sep_0$ and by induction on $k$. All other cases follow immediately using the same arguments. For the base case $k=1$ we apply \Cref{thm:connectedsum} to obtain
\begin{align}
\label{nearly-simplex}
\sep_0 (\altmanifold) = \sep_0 (\manifold) + \frac{1}{(d+2)(d+3)} - \frac{1}{n+1} + 	\frac{1}{n+2}.
\end{align}
Let $\altmanifold_k$ be obtained from $\manifold$ by stacking $k$ times, and $\altmanifold_{k+1}$ by stacking once in $\altmanifold_k$. Applying \Cref{thm:connectedsum} to $\altmanifold_k$ (that is, using \Cref{nearly-simplex}) and using the inductive hypothesis we obtain 

\begin{align*}
\sep_0 (\altmanifold_{k+1}) &= \sep_0 (\altmanifold_k)+ \frac{1}{(d+2)(d+3)} - \frac{1}{n+k+1} + 	\frac{1}{n+k+2} \\
& = \sep_0 (\manifold) + \frac{k +1 }{(d+2)(d+3)} - \frac{1}{n+1} + 	\frac{1}{n+k+2}.
\end{align*}
\end{proof}

\begin{remark}
  \label{rem:oned}
The proof of~\Cref{thm:connectedsum} still works for $d=1$, except for the fact that, since $0=d-1$, the contributions for the cases $i=0$ and $i=d-1$ are added yielding
\begin{align*}
 {\sep}_0(\manifold_1 \# \manifold_2) &=  \sep_0(\manifold_1) + \sep_0(\manifold_2) +  
 \frac{2}{3} - \frac{2}{n_1 +1}- \frac{2}{n_2 +1} +\frac{2}{n_1+n_2-2}.
\end{align*}
In particular, for a cycle $C_m$ with $m$ vertices we have that 
\begin{align*}
\sep_0(C_m)= 
\frac{m}6 + \frac2{m+1} -1 =
\frac{(m-2)(m-3)}{6(m+1)}.
\end{align*}
\end{remark}

\subsection{(Nearly)-neighborly manifolds}
\label{ssec:nneighborly}

A $d$-complex is called \emph{$k$-neighborly} if its $(k-1)$-skeleton is complete. This is equivalent to any of the following two equalities:
\[
f_{k-1} = \binom{f_0}{k} 
\quad
\Leftrightarrow
\quad
h_k = \binom{h_1 + k - 1}k  
\] 
Every complex is $1$-neighborly. The only $\lfloor d/2\rfloor + 1$-neighborly sphere is the boundary of a $d+1$-simplex, and $\lfloor d/2\rfloor$-neighborly manifolds are simply called \emph{neighborly}.

Observe that any $k$-neighborly $d$-complex $C$ has $\tilde{\beta}_{j} (C) = 0$ for all $j \leq k-2$, because $\tilde{\beta}_j$ only depends on the $(j+1)$-skeleton, and $(j+2)$-neighborly means that the $(j+1)$-skeleton coincides with that of a simplex, which is $j$-connected. Therefore the following statement necessarily holds.

\begin{proposition}[Bagchi, Datta, Lemma 3.9 of \cite{Bagchi14StellSpheresTightness}]
\label{prop:neighborly}
Let $C$ be a $k$-neighborly simplicial complex of dimension $d$, then $\sep_i (C) = 0$ for all $i \in \{0,\dots,k-2 \}$.

Conversely, if $\sep_i (C) = 0$ then $C$ is (at least) $(i+2)$-neighborly.
\end{proposition}

More generally, assume that $C$ is close to $(i+2)$-neighborly with only very few $(i+1)$-faces missing. Then $\sep_i (C)$ can only take few distinct values corresponding to the few possibilities for the isomorphism type of the $(i+1)$-skeleton. Here we explore the case $i=0$.

Assume only one edge is missing from the $1$-skeleton of $C$, then $\sep_0 (C) = \frac{1}{(n+1){n \choose 2}}$. The only non-zero contribution comes from the induced subcomplex on the two vertices that do not share an edge. If two edges are missing, these can either be disjoint or they meet in an edge. In the former case we have $\sep_0 (C) =  \frac{2}{(n+1){n \choose 2}}$ in the latter we have  $\sep_0 (C) = \frac{1}{n+1} \left (2/{n \choose 2} + 1/{n \choose 3} \right )$. 
Recall that, for at most $n - d - 2$ edges missing, \Cref{thm:nearlyneighb} provides a bound for the value of $\sep_0$.

\section{The $\sep$-vector of spheres}
\label{sec:spheres}

In this section we refine results stated in previous sections in the case that the simplicial complexes in question are triangulations of $d$-spheres.

\subsection{Euler, Dehn-Sommerville and Alexander relations}
\label{sec:DSeq}
Let $C$ be a simplicial complex. The Euler-Poincar\'e formula for $C$
\[
\chi(C):= \sum_{k=0}^\infty (-1)^k f_k(C) =  \sum_{k=0}^\infty (-1)^k \beta_k(C)
\]
or its reduced version
\[
\tilde \chi(C):= \sum_{k=-1}^\infty (-1)^k f_k(C) =  \sum_{k=-1}^\infty (-1)^k \tilde \beta_k(C)
\]
translates into an expressions for the alternating sum of entries in the $\sep$-vector in terms of the $f$-vector or the $h$-vector (see \Cref{ssec:notation} for a definition of the $h$-vector and some background). They imply the following variation of Lemma 2.5 of \cite{Bha2016} (originally stated for the $\sigma$-vector, into the $\sep$-vector).

\begin{lemma}[Bagchi~\protect{\cite[Lemma 2.5]{Bha2016}}]
  \label{lem:general}
  Let $C$ be a simplicial complex of dimension $d$ with $f$-vector $f(C) = (f_{-1},f_0,f_1, \ldots , f_d)$
  and $\sep$-vector $\sep (C) = (\sep_0,\sep_1, \ldots , \sep_d)$. Then
  \begin{equation}
    \label{eq:general}
    \sum \limits_{i=-1}^{d} (-1)^i {\sep}_i(C) = \sum \limits_{k=-1}^{d} \frac{(-1)^k}{k+2} f_k .
  \end{equation}
\end{lemma}

Plugging in the expressions relating the $f$-vector to the $h$-vector one easily concludes:

\begin{corollary}[Bagchi~\protect{\cite[Lemma 2.5]{Bha2016}}]
  \label{cor:hvec}

  Let $C$ be a $d$-dimensional simplicial complex with $h$-vector $h(C) = (h_0,h_1, \ldots , h_{d+1})$. Then
  \[
   \sum \limits_{i=-1}^{d} (-1)^i \sep_i (C) =  \sum \limits_{k=0}^{d+1} \frac{(-1)^{k+1}}{(d+2){d+1 \choose k}} h_k. 
   \]
\end{corollary}

Recall that when $C$ is topologically a $d$-sphere the $h$-vector entries satisfy the Dehn-Sommerville equations $h_k= h_{d+1-k}$. For even dimensional spheres, \Cref{cor:hvec} simply states $\sum_{i=-1}^{d} (-1)^i \sep_i (C) = 0$. This can also be deduced from Alexander duality, or, more precisely, is implied by the following statement.

\begin{lemma}[Bagchi and Datta \protect{\cite[Lemma 2.2]{Bagchi14StellSpheresTightness}}]
\label{lemma:duality}
  Let $\sphere$ be a triangulated $d$-sphere, then
  \[ 
  \sep_i = \sep_{d-i-1}, \ \ 
  \text{for all } \ -1 \leq i \leq d.
  \]
\end{lemma}

\begin{proof}
Let $V$ be the vertex set of $\sphere$.
By Alexander duality, the contribution of each $W$ to $\sep_i$ equals the contribution of $V\setminus W$  to $ \sep_{d-i-1}$.
\end{proof}

For odd-dimensional spheres we can combine the Dehn-Sommerville relations and Alexander duality in order to obtain yet another special case of \Cref{lem:general}.

\begin{corollary}
  \label{cor:oddsphere}
  Let $\sphere$ be a triangulation of the $d$-sphere, $d$ odd. Then
   \begin{align}
   \sum \limits_{i=-1}^{d} (-1)^i \sep_i (C)
   &=
   (-1)^{\frac{d-1}{2}} \sep_{\frac{d-1}{2}} + 2 \sum \limits_{i=-1}^{\frac{d-3}{2}} (-1)^i \sep_i (C) \\
   &= 
   \frac{1}{d+2} \left ( 2 \sum \limits_{k=0}^{\frac{d-1}{2}} \frac{(-1)^{k+1} h_k}{{d+1 \choose k}} + \frac{(-1)^{\frac{d-1}{2}} h_{\frac{d+1}{2}}}{{d+1 \choose (d+1)/2}} \right ). 
   \end{align}
\end{corollary}

\begin{proof}
Since $\sphere$ is a triangulation of the $d$-sphere, we  have $h_k = h_{d+1-k}$, $0 \leq k \leq d+1$, and $h_0=1$. Moreover, we have $ \sep_j =  \sep_{d-1-j}$, $0 \leq j \leq d-1$, by 
\Cref{lemma:duality}.
\end{proof}

For dimension three this takes the following simple form.

\begin{corollary}
  \label{cor:3d}
  For every triangulated $3$-sphere $\sphere$ we have ${\sep}_0 = {\sep}_2$ and
\begin{equation}
    2 {\sep}_0 - {\sep_{1}}  = \frac{h_1(h_1+1)}{10(h_1+5)}  - \frac{h_2}{30} . 
\end{equation}
Equivalently, in terms of the $f$- and $g$-vector
\begin{equation}
\label{eq:sep_for_3spheres_fg}
    2 {\sep}_0 - {\sep_{1}}  
    = \frac{f_0^2-4f_0+5}{5(f_0+1)}  - \frac{f_1}{30},
    \qquad 
    2 {\sep}_0 - {\sep_{1}}  
    = \frac{g_1(g_1+1)}{15(g_1+6)}   - \frac{g_2}{30}.
\end{equation}
\end{corollary}

\subsection{Stacked spheres}
\label{ssec:stacked}
A closed $d$-manifold is called \emph{$k$-stacked} if it is the boundary of a $(d+1)$-manifold whose interior faces all have dimension greater than $d-k$. The only $0$-stacked manifold is the boundary of a $(d+1)$-simplex, and $1$-stacked manifolds are simply called stacked.
In this section we explore bounds and exact values for the entries of the $\sep$-vector for $k$-stacked $d$-spheres.
Stackedness is stronger (and more interesting) for small values of $k$.

For the $\sep$-vector we have the following result first proven by Bagchi \cite[Lemma 3]{Bagchi15TightnessCrit}. 

\begin{theorem}
  \label{thm:kstacked}
  Let $\sphere$ be a $k$-stacked $d$-sphere. Then $\sep_i (\sphere) = 0$ for $k \leq i \leq d-k-1$.
\end{theorem}

\begin{remark}
\label{rem:k_stellated}
Adiprasito \cite{Adiprasito18Lefschetz} (see \cite[Corollary 6.5]{Adiprasito15ToricChordality} for the case of simplicial polytopal spheres) proves that for every induced subcomplex $\sphere[W]$ of a triangulated $d$-sphere $\sphere$ and for every $k\leq(d-1)/2$,
\[
\tilde \beta_{k}(\sphere[W]) \le g_{k+1}(\sphere).
\]
The fact that triangulated $k$-stacked $d$-spheres, $k \le (d-1)/2$, have $g_{i+1}(\sphere)=0$, $k \leq i \leq d-k-1$, makes \Cref{thm:kstacked} a special case of this statement.
\end{remark}

Spheres which are $1$-stacked are often simply referred to as {\em stacked}. 
Stacked $d$-spheres are precisely those triangulated $d$-spheres which can be obtained from the boundary of the $(d+1)$-simplex by a number of bistellar $(1,d+1)$-moves, or, equivalently, which can be written as the connected sum of boundaries of the $(d+1)$-simplex.
This implies that the stacking order is a perfect elimination order with in-degree sequence $(0,1,2,\dots,d+1,\dots,d+1)$.
Thus if $\sphere$ is a stacked sphere, we have $\sep_{-1}(\sphere)=\sep_d(\sphere) =\frac1{f_0+1}$ and, due to \Cref{cor:stacking},

\begin{equation*}
 \sep_0 (\sphere) =  \sep_{d-1} (\sphere)
    =  \frac{n-d-2}{(d+2)(d+3)} - \frac{1}{d+3} + \frac{1}{n+1} 
    =  \frac{n-2d-4}{(d+2)(d+3)} + \frac1{n+1}.
\end{equation*}

Recall that stacked spheres maximize $\sep_0$ among all normal pseudo-manifolds with a given dimension and number of vertices, see \Cref{thm:BDSS}. 
Moreover, \Cref{thm:kstacked} implies that they have $\sep_i = 0$ for all other entries. In particular, the $\sep$-vector of a stacked sphere is completely determined by its dimension and its number of vertices, and it is extremal among all normal-pseudo-manifolds of a given dimension and number of vertices: $\sep_0$ is maximal and every other $\sep_i$ is minimal.

\subsection{Nearly stacked spheres}
\label{ssec:nstacked}

We now focus on complexes which are not stacked, but {\em nearly stacked}, meaning that $g_2=1$ instead of $0$. That is, we call a $d$-sphere with $n$ vertices nearly stacked if it has $f_1(\sphere) = (d+1)n - {d+2 \choose 2}+1$ edges, one more than the number of edges of a stacked sphere \cite{Kalai87RigidityLBT}. 
These spheres have the following complete characterisation by Nevo and Novinsky \cite{NevoNovinsky11}: a $d$-sphere $\sphere$ satisfies $g_2(\sphere) = 1$ if and only if it is obtained  either from $\partial \Delta_i \ast \partial \Delta_j$, $i+j = d+1$, $i,j \geq 2$, or from $C_k \ast \partial \Delta_{d-1}$ by an arbitrary number of stacking operations. Here $\partial \Delta_i$ denotes the boundary of an $i$-simplex and $C_k$ is the boundary of a $k$-gon.%
\footnote{Nevo and Novinsky state this result in a more specialized version for simplicial $(d+1)$-polytopes, but they prove it for homology $d$-spheres.}

In the following two lemmas we compute the $\sep$-vectors of these two minimal cases.

\begin{lemma}
\label{lem:simplexjoin}
Let $\sphere=\partial \Delta_i \ast \partial \Delta_j$, $i,j \geq 2$, $d=i+j-1$. Then all entries of the $\sep$-vector of $S$ are $0$ except
\begin{align*}
& \sep_{-1}= \sep_{d}= \frac{1}{d+4} \\
& \sep_{i-1}= \sep_{j-1}=\frac{1}{(d+4)\binom{d+3}{i+1}}, \text{ if } i \neq j \\
& \sep_{i-1}=\frac{2}{(d+4)\binom{d+3}{i+1}}, \text{ if } i =j.
\end{align*}
\end{lemma}

\begin{proof}
Let $V_1$ and $V_2$ be the subsets of vertices of $\partial \Delta_i$ and $ \partial \Delta_j$ respectively. Since induced subcomplexes of the join are joins of the induced subcomplexes and since all induced subcomplexes of $\partial \Delta_i$ are contractible except for the empty one and the full one, the only $W$'s that contribute to the $\sep$-vector are $W\in \{\emptyset, V_1,V_2, V_1\cup V_2\}$, which add respectively to $\sep_{-1}$, $\sep_{i-1}$, $\sep_{j-1}$,  and $\sep_{i+j-1}$, as stated.
\end{proof}

\begin{lemma}
\label{lem:cyclejoin}
Let $\sphere_m=C_{m} \ast \partial \Delta_{d-1}$, with $d\ge 4$. Then all entries of the $\sep$-vector of $\sphere_m$ are $0$ except
\begin{align*}
 \sep_{-1}= \sep_{d}&= \frac{1}{m+d+1} \\
 \sep_{0}= \sep_{d-1} &= \frac{m-3}{(d+3)(d+2)} + \frac1{m+d+1} + \frac1{(d+1)\binom{d+m+1}{d+1}}
- \left(\frac1{d+4} + \frac1{(d+1)\binom{d+4}{3}}\right)
\\
\sep_{1}= \sep_{d-2}&= \frac{1}{(m+d+1)\binom{m+d}{m}}.
\end{align*}
For $d=3$ the same holds except the value of $\sep_1$ is multiplied by a factor of $2$.
\end{lemma}

\begin{proof}
Using again that the only $W$'s that can contribute are those that contain either all or none of the vertices of $ \Delta_{d-1}$, all values are straightforward except for $\sep_0$ and $\sep_{d-1}$, which coincide by Alexander duality. Hence, we focus on $\sep_0$ for the rest of the proof.

First observe that, since any vertex of $\partial \Delta_{d-1}$ is joined to all other vertices of $\sphere_m$ by an edge, for a subcomplex induced by a subset of vertices of $\sphere_m$ to have non-trivial $0$-homology it must only contain vertices of $C_m$. Let $v$ be one vertex of $\sphere_m$ which belongs to the cycle $C_m$. Then if $V$ is the vertex set of $\sphere_{m-1}$, we let $V \cup \{v\}$ be the vertex set of $\sphere_m$, and let $x,y$ be the neighbors of $v$ in $C_m$.  By \Cref{thm:ground_set} we can consider $\sphere_{m-1}$ and $\sphere_{m}$ as having the same ground set $V\cup \{v\}$ and obtain
\begin{align*}
\sep_0(\sphere_{m-1}) 
&= \frac{1}{m+d+1} \sum_{X \subseteq V} \left( \frac{\tilde \beta_0(\sphere_{m-1}[X]) }{{ m+d \choose |X|}} + \frac{\tilde \beta_0(\sphere_{m-1}[X \cup v]) }{{ m+d \choose |X \cup v|}} \right)\\
&= \frac{1}{m+d+1} \sum_{X \subseteq V} \left( \frac{\tilde \beta_0(\sphere_{m-1}[X]) }{{ m+d \choose |X|}} + \frac{\tilde \beta_0(\sphere_{m-1}[X]) }{{ m+d \choose |X \cup v|}} \right)
\end{align*}

Then, the difference between $\sep_0$ of $\sphere_m$ and of $\sphere_{m-1}$ is given by

\begin{align*}
\sep_0(\sphere_m) - \sep_0(\sphere_{m-1}) = &\frac{1}{m+d+1} \sum_{X \subseteq V}  \frac{\tilde \beta_0(\sphere_{m}[X]) - \tilde \beta_0(\sphere_{m-1}[X]) }{{ m+d \choose |X|}} + \\
+&
\frac{1}{m+d+1} \sum_{X \subseteq V} \frac{\tilde \beta_0(\sphere_{m}[X \cup v]) - \tilde \beta_0(\sphere_{m-1}[X \cup v]) }{{ m+d \choose |X \cup v|}}.
\end{align*}

The numerator of the first sum is equal to  $1$ if  $ \{x,y\} \subseteq X$ and $X \ne V(C_m)$, otherwise it is $0$.
The numerator in the second sum is equal to $1$ if $X \cap \{x,y\} = \emptyset$ and $X \neq \emptyset$, otherwise it is $0$. Thus the above equation turns into
\begin{align*}
\sep_0(\sphere_m) - \sep_0(\sphere_{m-1}) & = \frac{1}{m+d+1} \sum_{\substack{ X \subsetneq V \\ \{x,y\}\subseteq X}}  \frac{1 }{{ m+d \choose |X|}} + \frac{1}{m+d+1} \sum_{\substack{X \subseteq V \setminus \{x,y\} \\ X \neq \emptyset}} \frac{1 }{{ m+d \choose |X|+1}} \\
& = \frac{1}{m+d+1} \left(\sum_{k=0}^{m-4} \frac{{m-3 \choose k}}{{m + d \choose k+2 }} +  \sum_{k=1}^{m-3} \frac{{m-3 \choose k}}{{m + d \choose k+1 }}\right) 
 = \frac{1}{m+d+1} \sum_{k=1}^{m-3} \frac{{m-2 \choose k}}{{m + d \choose k+1 }} \\
& = \frac{1}{m+d+1}  \left(  \sum_{k=0}^{m-2} \frac{{m-2 \choose k}}{{m + d \choose k+1 }} - \frac{1}{m+d} - \frac{1}{{m+d \choose d+1}} \right)\\
& \overset{*}{=}	 \frac{1}{m+d+1}  \left(  \frac{m+d+1}{(d+3)(d+2)} - \frac{1}{m+d} - \frac{1}{{m+d \choose d+1}} \right)\\
& = \frac{1}{(d+3)(d+2)} - \frac{1}{m+d+1}  \left(  \frac{1}{m+d} + \frac{1}{{m+d \choose d+1}} \right).
\end{align*}

The equality $\overset{*}{=}$ is a consequence of \Cref{lemma:binomial_quotients} with $n = m-2$, $a=1$ and $b=d+2$.

Let us now introduce 
\[
f(m):=\frac{m+d+1}{(d+3)(d+2)} + \frac1{m+d+1} + \frac1{(d+1)\binom{d+m+1}{d+1}}
\]
and observe that 
\begin{align*}
f(m)&-f(m-1) =\\
&=\frac{1}{(d+3)(d+2)} - \frac1{(m+d+1)(m+d)} + \left(\frac1{(d+1)\binom{d+m+1}{d+1}} - \frac1{(d+1)\binom{d+m}{d+1}}\right)\\
&=\frac{1}{(d+3)(d+2)} - \frac1{(m+d+1)(m+d)} - \left(\frac{\binom{d+m+1}{d+1}-\binom{d+m}{d+1}}{(d+1)\binom{d+m+1}{d+1} \binom{d+m}{d+1}}\right)\\
&=\frac{1}{(d+3)(d+2)} - \frac1{(m+d+1)(m+d)} - \left(\frac{\frac{d+m+1}{m} \binom{d+m}{d+1}-\binom{d+m}{d+1}}{(d+1)\binom{d+m+1}{d+1} \binom{d+m}{d+1}}\right)\\
\end{align*}
\begin{align*}
&=\frac{1}{(d+3)(d+2)} - \frac1{(m+d+1)(m+d)} -  \frac1{m\,\binom{d+m+1}{d+1}}\\
&=\frac{1}{(d+3)(d+2)} - \frac1{(m+d+1)(m+d)} -  \frac1{(m+d+1)\binom{m+d}{d+1}}\\
&=\sep_0(\sphere_m) - \sep_0(\sphere_{m-1}).
\end{align*}

Since $\sep_0(\sphere_3) = 0$ by the previous lemma, for every $m\ge 3$ we have
\begin{align*}
\sep_0(\sphere_m) 
= \sum_{t=4}^m \left ( \sep_0(\sphere_t) - \sep_0(\sphere_{t-1}) \right )
= \sum_{t=4}^m  \left (f(t) - f(t-1) \right )= f(m) - f(3),
\end{align*}
as stated.
\end{proof}

Stacking operations only change two (pairs of) entries in the $\sep$-vector: $\sep_0=\sep_{d-1}$ and $\sep_1=\sep_{d-2}$. The following two statements give the exact minimum and maximum value for these entries among all nearly stacked spheres. The other entries are as described in \Cref{lem:simplexjoin,lem:cyclejoin}.

\begin{theorem}($\sep_0$ of nearly stacked spheres.)
  \label{thm:g21_sep0}
  Let $\altsphere$ be an $n$-vertex triangulation of the $d$-sphere with $g_2 (\altsphere) = 1$ and let
  \[
  c_d(n) : = \frac{n-d-3}{(d+2)(d+3)} +\frac{1}{n+1} - \frac{1}{d+4}.
  \]
  Then
  \begin{align*}
     c_d(n) - \frac1{(d+1)}\left (\frac1{\binom{d+4}{3}} - \frac1{\binom{n+1}{d+1}}\right)
     \le \sep_0 (\altsphere) = \sep_{d-1}(\altsphere) \leq c_d(n).\\ 
  \end{align*}
Equality in the lower bound is attained only by $C_{n-d}*\partial \Delta_{d-1}$.
Equality in the upper bound is attained only by applying $(n-3-d)$ stacking operations to $\partial \Delta_i*\partial \Delta_j$ for any $i,j\ge 2$ with $i+j=d-1$.
\end{theorem}

\begin{proof}
The spheres of the form $\partial \Delta_i *\partial \Delta_j$ all have $\sep_0=0$ and $d+3$ vertices. By \Cref{cor:stacking}, all $d$-spheres with $n$ vertices obtained by stacking $n-d-3$ times have 
\[
\sep_0= \frac{n-d-3}{(d+2)(d+3)} +\frac{1}{n+1} - \frac{1}{d+4} .
\]
Observe that $d$-spheres that are stacked over $C_{3} *\partial \Delta_{d-1}$ are a particular case of this.

On the other extreme, by substituting $m=n-d$ in \Cref{lem:cyclejoin}, the $d$-sphere $C_{n-d} *\partial \Delta_{d-1}$ satisfies
\begin{align*}
 \sep_{0} (C_{n-d} *\partial \Delta_{d-1}) = \frac{n-d-3}{(d+3)(d+2)} + \frac1{n+1} + \frac1{(d+1)\binom{n+1}{d+1}}
- \left(\frac1{d+4} + \frac1{(d+1)\binom{d+4}{3}}\right) .
\\
\end{align*}

Between these two cases are the intermediate ones where we stack $k$ times over $C_{m} *\partial \Delta_{d-1}$, for $m+k=n-d$. 
We only need to show that the value we obtain for $\sep_0$ monotonically increases with $k$. 
For this it suffices to show that if $\sphere'_m$ denotes the sphere obtained by a single stacking of $C_{m-1} *\partial \Delta_{d-1}$, we have $\sep_0(\sphere'_{m}) > \sep_0(C_{m} *\partial \Delta_{d-1})$.
This follows from the fact that, by \Cref{cor:stacking},
\[
\sep_0(\sphere'_m) - \sep_0(C_{m-1} *\partial \Delta_{d-1}) = \frac{1}{(d+2)(d+3)} - \frac{1}{(m+d)(m+d+1)},
\]
while, as seen in the proof of \Cref{lem:cyclejoin},
\[
\sep_0(C_{m} *\partial \Delta_{d-1}) - \sep_0(C_{m-1} *\partial \Delta_{d-1}) 
= \frac{1}{(d+3)(d+2)} - \frac{1}{m+d+1}  \left(  \frac{1}{m+d} + \frac{1}{{m+d \choose d+1}} \right).
\]
\end{proof}

\begin{theorem}($\sep_1$ of nearly stacked spheres.)
  \label{thm:g21_sep1}
  Let $\altsphere$ be an $n$-vertex triangulation of the $d$-sphere with $g_2 (\altsphere) = 1$. Then
  \begin{align*}
    \label{eq:d3}
     \frac{1}{2{n+1 \choose 4}} &\leq \sep_1 (\altsphere) \leq \frac{1}{70}, &&  \text{for $d=3$},\\ 
     \frac{1}{5{n+1 \choose 5}} &\leq \sep_1 (\altsphere)=\sep_2(\altsphere) \leq \frac{1}{280} &&  \text{for $d=4$,}\\ 
     0 &\leq \sep_1 (\altsphere) =\sep_{d-2}(\altsphere) \leq \frac{1}{4 {d+4 \choose 4}}. &&  \text{for $d\ge 5$.}\\ 
  \end{align*}
Equality in the lower bound is attained only by $C_{n-d} *\partial \Delta_{d-1}$ ($d=3,4$) and $\partial \Delta_{i} *\partial \Delta_{j}$ ($i,j\ge 3$, $d \ge 5$). 
Equality in the upper bound is attained only by applying $(n-3-d)$ stacking operations to $C_3 *\partial \Delta_{d-1}$.
\end{theorem}

\begin{proof}
By~\cite[Theorem 1.3]{NevoNovinsky11}, every sphere with $g_2=1$ can be obtained from either $\partial \Delta_i \ast \partial \Delta_j$, $i+j = d+1$, $i,j \geq 2$, or from $C_k \ast \partial \Delta_{d-1}$ or stacking, and \Cref{lem:simplexjoin,lem:cyclejoin} give us the $\sep$-vector in this case.

Since $\sep_1$ does not change under stacking, \Cref{lem:simplexjoin,lem:cyclejoin} directly produce the possible values for $\sep_1$. Moreover, for $d=3,4$ we have that the only case in \Cref{lem:simplexjoin} is the special case $C_3 \ast \partial \Delta_{d-1}$ of \Cref{lem:cyclejoin}. It follows that
\begin{align*}
\sep_1(\altsphere) = &\frac{2}{(m+4)\binom{m+3}3} = \frac{1}{2\binom{m+4}4}, && \text{if $d=3$},\\
\sep_1(\altsphere) = &\frac{1}{(m+5)\binom{m+4}4} = \frac{1}{5\binom{m+5}5}, && \text{if $d=4$}.
\end{align*}
The extremal cases are $m=3$ and $m=n-d$, which give
\begin{align*}
\frac{1}{2\binom{n+1}4} & \le \sep_1(\altsphere) \le \frac{1}{70}, && \text{if $d=3$},\\
\frac{1}{5\binom{n+1}5} &\le \sep_1(\altsphere) \le \frac{1}{280}, && \text{if $d=4$}.
\end{align*}

For $d>4$ the same arguments apply except \Cref{lem:simplexjoin} is no longer a special case of \Cref{lem:cyclejoin}. In fact, what now happens is that both Lemmas give the same upper bound, obtained by $\sphere= \partial \Delta_2 \ast \partial \Delta_{d-1} =C_3 \ast \partial \Delta_{d-1}$. The lower bound is zero, obtained by  $\partial \Delta_i \ast \partial \Delta_{j} $ with $i,j > 2$. That is, we have
\begin{align*}
0 \le \sep_1 (\sphere) & \le \frac{1}{(d+4)\binom{d+3}{3}} = \frac{4}{\binom{d+4}4},  \quad \text{for $d>4$}.
  \end{align*}

\end{proof}

\begin{remark}
  A classification for triangulated $d$-spheres with $g_2 = 2$ is known due to Zheng \cite{Zheng17g2eq2}. All such triangulated $d$-spheres are polytopal. However, the classification itself is significantly more involved than the one for $g_2 = 1$.
\end{remark}

\section{The $\sep$-vector of $3$-spheres and the $\mu$-vector of $4$-manifolds}
\label{sec:conjecture}

In this section we take a closer look at the $\sep$-vector of $3$-spheres and the $\mu$-vector of $4$-manifolds. Thus, before we start, let us first recall the definition of the $\mu$-vector and its relation with the $\sep$-vector. Let $\manifold$ be  a manifold, then \Cref{thm:bagchi-intro} states that
$ \tilde\beta_i(\manifold) \le \mu_i(\manifold) :=\sum_{v\in \Vertices} \sep_{i-1} \lk_\manifold (v)$,
with equality if and only if $\manifold$ is $\F$-tight. 

This statement provides two reasons for studying the $\sep$-vector:

\begin{itemize}
  \item The $\sep$- and $\mu$-vectors can be used to check whether a simplicial complex is tight. This motivated Bagchi and Datta's work in \cite{Bagchi14StellSpheresTightness}.
  \item Upper bounds on $\sep$ of the links in terms of their $f$-vectors can provide lower bounds on the $f$-vector of triangulations of a manifold in terms of its Betti numbers. This technique was successfully used by Burton, Datta, Singh and Spreer in \cite{Burton14SepIndex2Spheres} to obtain an alternative proof for \Cref{thm:lss} below.
\end{itemize}
 
The case of $4$-manifolds and $3$-spheres is especially interesting because here the $\sep$-vector of the links (hence the $\mu$-vector of the manifold) is determined by $\sep_0$. In particular, it is independent of the choice of field. 

In \Cref{ssec:3spheres} we give experimental data on the $\sep$-vector for $3$-spheres with up to $10$ vertices. 
In \Cref{ssec:tightness} we review the concept of $\F$-tightness. In \Cref{ssec:4mflds} we discuss a lower bound for triangulations of $4$-manifolds with automorphism group acting transitive on their vertices based on \Cref{conj:general-intro}.

\subsection{The $\sep$-vector of $3$-spheres}
\label{ssec:3spheres}

By \Cref{cor:3d}, the $\sep$-vector of a sphere $\sphere$ of dimension up to three and with a given $f$-vector is determined by $\sep_0$; that is, choosing a field and looking at the respective reduced Betti numbers reduces to counting connected components of induced subcomplexes of $\sphere$, see \Cref{eq:sep_for_3spheres}. 
Since, moreover,  $\sep_1$ monotonically depends on $\sep_2 = \sep_0$, this value seems to be a good parameter to estimate the ``combinatorial complexity'' of a triangulated $3$-sphere with prescribed $f$-vector. In higher dimensions we can still use $\sep_0$ as a measure of complexity, but it no longer captures the entire information of the $\sep$-vector. 

\Cref{fig:data,fig:data2} present the values of $\sep_0$ for all triangulated $3$-spheres with $g_1 + 5 = f_0 \in \{7,\dots, 10\}$, plotted in terms of $g_2=f_1 - (4f_0 -10)$. Observe that once the number of vertices is fixed, $g_2$ completely characterizes the $f$-vector and it ranges between $0$ (stacked $3$-spheres) and $\binom{g_1+1}{2}$ (neighborly spheres). We do not include plots for $f_0\le 6$ because those have $g_2\in \{0,1\}$ and are completely classified, as described in detail in \Cref{ssec:nstacked}. In particular, there is only one sphere for each $f$-vector, the Billera-Lee sphere.

The enumeration of $3$-spheres up to 10 vertices is due to \cite{Altshuler76CombMnf9VertAll,Lutz08ThreeMfldsWith10Vertices}, and our data is taken from \cite{Lutz08ManifoldPage}. 
The data-points corresponding to Billera-Lee spheres (the maximum) are marked by red triangles.

In addition, using the classification of simplicial $4$-polytopes up to $10$ vertices due to Firsching \cite{FirschingRealizability2017}, we mark every $\sep_0$-value of a simplicial polytopal $3$-sphere with a square. The data for this experiment is taken from \cite{Firsching18PolytopePage}.

These calculations produce additional insight into the nature of the $\sep$-vector:

  The data gives an idea of the \emph{range} of values to expect for the entries of the $\sep$-vector for a set of triangulations of the $3$-sphere with fixed $f$-vector.

\begin{figure}[htbp]
    \includegraphics[width=.46\textwidth]{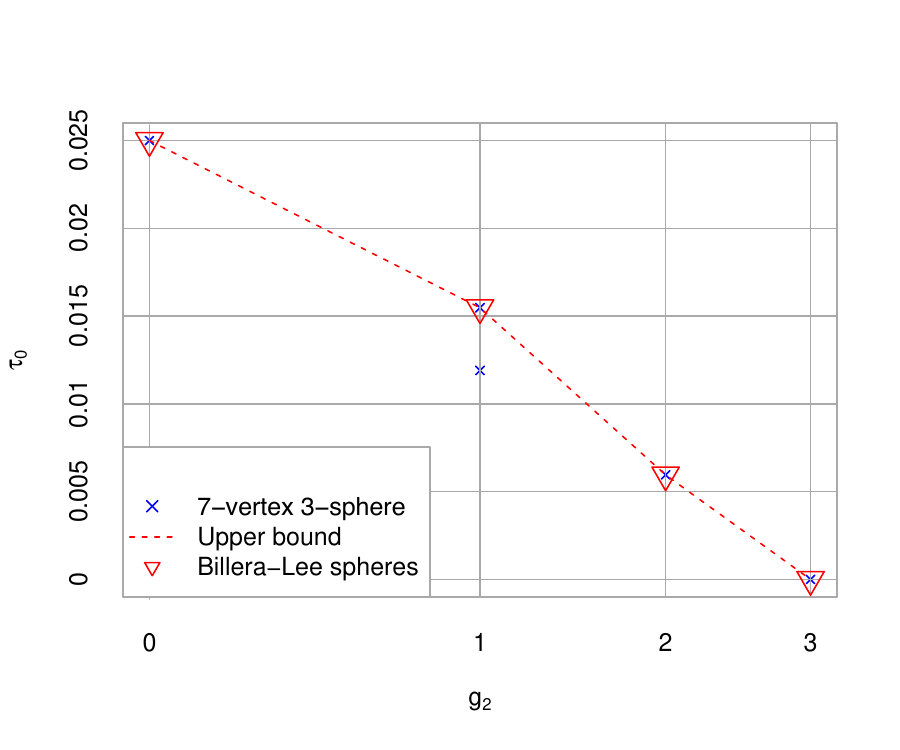} 
    \includegraphics[width=.53\textwidth]{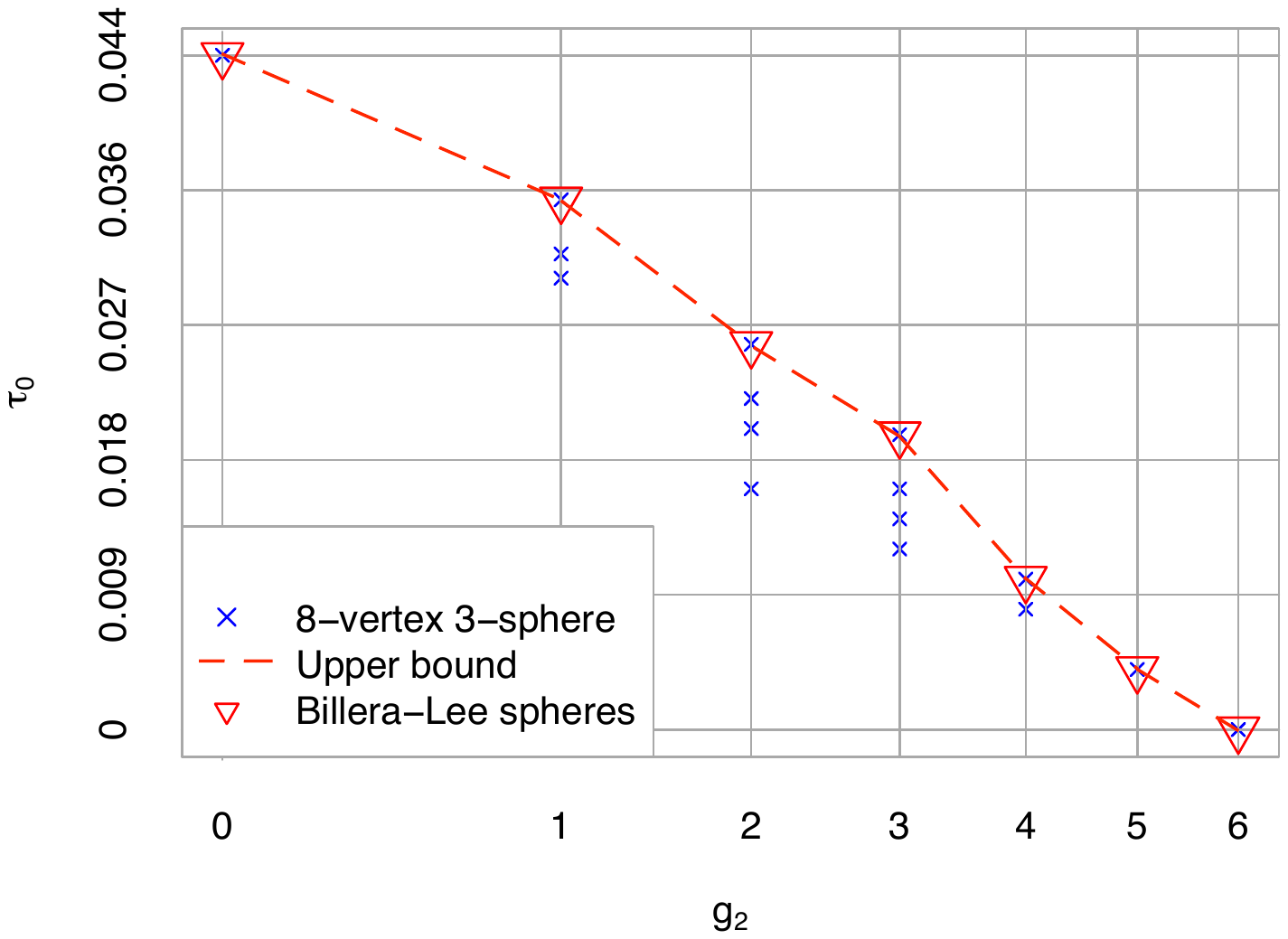} \\
  \caption{Values of $\sep_0$ for all triangulated $3$-spheres with $7$ and $8$ vertices with the upper bound from \Cref{conj:general-intro} and \Cref{thm:upperbound} (dashed line and triangles for the Billera-Lee spheres). All $\sep_0$-values can be realized by simplicial polytopal spheres. \label{fig:data} }
\end{figure}

\begin{figure}[phtb]
    \includegraphics[width=.9\textwidth]{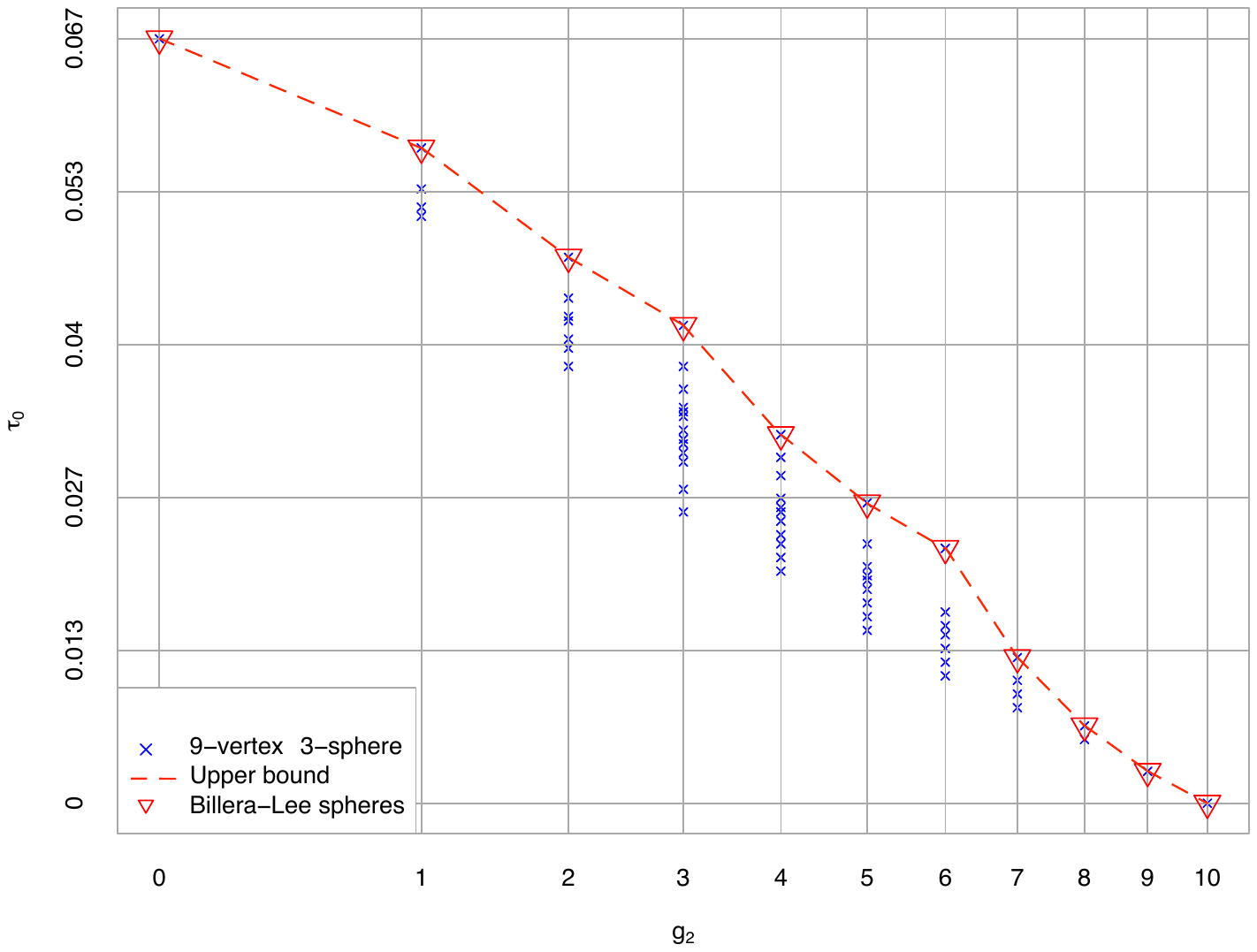} \\ \vskip -1cm
    \includegraphics[width=.9\textwidth]{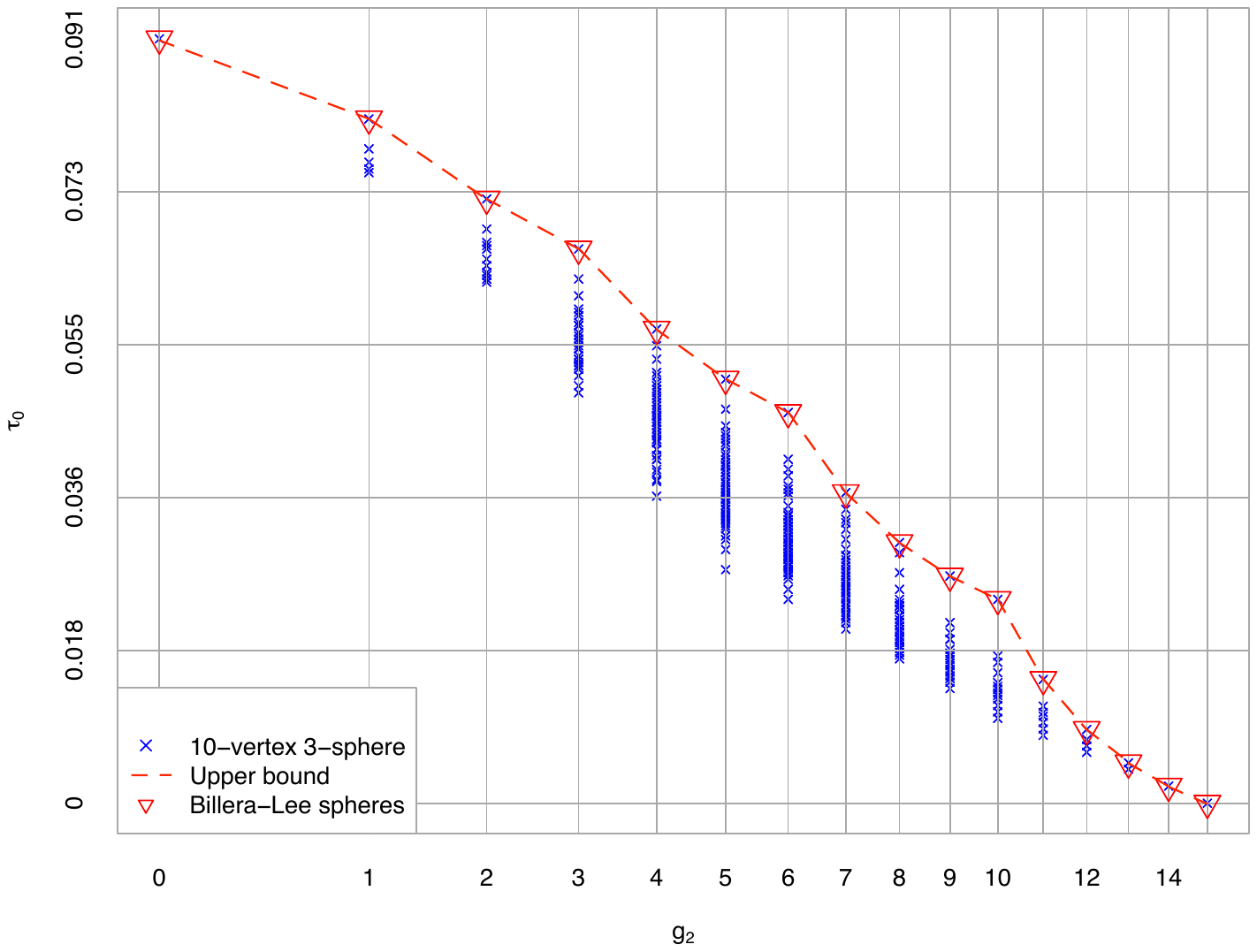} 
    \caption{Values of $\sep_0$ for all triangulated $3$-spheres with $9$ and $10$ vertices. Note that minimal values for $\sep_0$ are not always realizable by simplicial polytopal spheres. \label{fig:data2} }
\end{figure}

\begin{remark}
In general there are many spheres with both the same $\sep$- and $f$-vectors as a given Billera-Lee sphere. For example, there is more than one neighborly $3$-sphere and more than one stacked $3$-sphere for each number of vertices. By \Cref{prop:neighborly} and \Cref{thm:kstacked} these $3$-spheres all have the same $\sep$-vector as the corresponding Billera-Lee spheres. 
Even more, there are examples of $3$-spheres with the $\sep$-vector of a Billera-Lee sphere admitting perfect elimination orderings on their graphs with different in-degree sequences. (Recall that the $\sep$-vector only depends on the in-degree sequence as a multiset, not on its order.)

As an example, the $8$-vertex $24$-edge Billera-Lee sphere has a perfect elimination order with in-degree sequence $(0,1,2,3,4,5,5,4)$. On the other hand, the connected sum of two copies of the cyclic $4$-polytope $C_4(6)$ has a perfect elimination order with in-degree sequence $(0,1,2,3,4,5,4,5)$. (There are actually two  isomorphism types of such connected sums, none of them having the $1$-skeleton of a Billera-Lee sphere.)

\Cref{fig:blspheres} lists the number of $3$-spheres up to $10$ vertices with the $\sep$-vector of the corresponding Billera-Lee sphere.

\begin{figure}[htb]
  \begin{center}
  {\scriptsize
  \begin{tabular}{r||@{\hspace{0.03cm}}r|@{\hspace{0.03cm}}r|@{\hspace{0.03cm}}r|@{\hspace{0.03cm}}r|@{\hspace{0.03cm}}r|@{\hspace{0.03cm}}r|@{\hspace{0.03cm}}r|@{\hspace{0.03cm}}r|@{\hspace{0.03cm}}r|@{\hspace{0.03cm}}r|@{\hspace{0.03cm}}r|@{\hspace{0.03cm}}r|@{\hspace{0.03cm}}r|@{\hspace{0.03cm}}r|@{\hspace{0.03cm}}r|@{\hspace{0.03cm}}r|@{\hspace{0.03cm}}|@{\hspace{0.03cm}}r@{}}
    \toprule
    \diagbox[width=.7cm,height=0.7cm]{$f_0$}{$g_2$}&$0$&$1$&$2$&$3$&$4$&$5$&$6$&$7$&$8$&$9$&$10$&$11$&$12$&$13$&$14$&$15$&$\Sigma$ \\
    \midrule
    \midrule
    $5$ &$1$&	&	&	&	&	&	&	&	&	&	&	&	&	&	& & $1$\\
    \midrule
    $6$ &$1$&$1$	&	&	&	&	&	&	&	&	&	&	&	&	&	& 	& $2$\\
    \midrule
    $7$ &$1$&$1$	&$1$	&$1$	&	&	&	&	&	&	&	&	&	&	&	& 	&$4$\\
    \midrule
    $8$ &$3$&$3$	&$5$	&$2$	&$3$	&$5$	&$4$	&	&	&	&	&	&	&	&	& 	&$25$\\
    \midrule
    $9$ &$7$&$14$	&$29$	&$14$	&$44$	&$48$	&$26$	&$42$	&$111$	&$121$	&$51$	&	&	&	&	& 	&$507$\\
    \midrule
    $10$ &$30$&$82$	&$235$	&$108$	&$541$	&$600$	&$287$	&$1\,019$	&$2\,505$	&$2\,890$	&$993$	&$1\,938$ &$5\,175$	&$13\,241$&$14\,057$&$3\,677$ &$47\,378$\\

    \bottomrule
  \end{tabular}}
  \end{center}

  \caption{Number of isomorphism types of triangulated $3$-spheres with the same $\sep$-vector as the corresponding Billera-Lee sphere. \label{fig:blspheres}}

\end{figure}
\end{remark}

\subsection{Tightness} 
\label{ssec:tightness}

Tightness was defined in a geometric context  by Alexandrov in 1938
in \cite{Alexandrov38ClassClosedSurf} and it is a generalization of convexity. 
Loosely speaking, a manifold is tight in Alexandrov's sense if it does not exhibit any dents or holes other than the ones required by its topology.
For example, a sphere is tight if and only if it is the boundary of a convex ball.

The combinatorial version of tightness we are interested in  was introduced by K\"uhnel in 1995 in \cite{Kuehnel99CensusTight} (Banchoff in 1970 introduced an intermediate version for embedded polyhedral manifolds in \cite{Banchoff65TightEmb3DimPolyMnf,Banchoff71TPPTightNMnfWithBd}):

\begin{definition}
\label{defi:tight}
Let $C$ be a simplicial complex of dimension $d$ with vertex set $\Vertices$ and let $\F$ be a field. We say that $C$ is \emph{$\F$-tight} if for every $W\subseteq \Vertices$ the inclusion 
\[
i_W: C[W]\to C
\]
induces an injective homomorphism in homology for every dimension. 
$C$ is called \emph{tight} if it is tight for at least one field. Equivalently, if it is tight for some prime field $\F_p$; see, for instance, \cite[Lemma 2.8(b)]{BDSTight3Mflds}.
\end{definition}

\begin{remark}
  \label{rem:tight}
  According to \Cref{defi:tight} a simplicial complex is tight, if every homological feature appearing in an induced subcomplex persists in the entire complex. Examples of simplicial complexes which are not tight include connected complexes with missing edges: the induced subcomplex on a subset $W=\{v_1,v_2\}$ with $v_1$ and $v_2$ not spanning an edge does not inject in $0$-dimensional homology. More generally, in a tight simplicial complex $C$, the boundary of every minimal non-face of size $d+2$ must be a generator of the $d$-th homology group of $C$.  
\end{remark}

One source of interest for tight triangulations lies within the following conjecture. 

\begin{conjecture}[Lutz and K\"uhnel \cite{Kuehnel99CensusTight}] 
  \label{conj:LutzKuehnel}
  Tight triangulations of a manifold $\manifold$ minimize every entry of the $f$-vector among all triangulations of~$\manifold$.
\end{conjecture}

\Cref{conj:LutzKuehnel} is trivially true in dimension 2, and it has recently been proven in dimension 3 by Bagchi, Datta and Spreer \cite{Bagchi16Tight3Mflds}.
More generally, most examples of tight triangulations fall in one of two classes for which \Cref{conj:LutzKuehnel} can be shown to hold. In the remainder of this section we 
\begin{enumerate}[(a)]
  \item describe these two classes; 
  \item explain why \Cref{conj:LutzKuehnel} holds in these cases; and
  \item discuss how the $\sep$-vector of an example of a tight triangulation not falling in one of these two classes together with the upper bound from \Cref{conj:general-intro} produces a gap in the inequality $ \tilde\beta_i(\manifold) \le \mu_i(\manifold) :=\sum_{v\in \Vertices} \sep_{i-1} \lk_\manifold (v)$ from \Cref{thm:bagchi-intro}, yielding a surprising new viewpoint on the conjecture, see \Cref{ex:tight}.
\end{enumerate}

As explained in \Cref{rem:tight}, every tight triangulation of a $d$-manifold is $2$-neighborly and the converse is true in dimension two. In arbitrary dimensions, two classes cover most known examples of tight triangulations: (a) Every $(k+1)$-neighborly triangulation of a $2k$-manifold is tight (This follows from \cite[Corollary 4.7]{Kuehnel95TightPolySubm}). We call such triangulations \emph{tight of neighborly type} (this class includes all tight $2$-manifolds);
(b) Every $2$-neighborly and stacked triangulation is tight \cite{Kalai87RigidityLBT,Novik08SocBuchsMod}. We call such triangulations \emph{tight of stacked type} (this class includes all tight $3$-manifolds by 
\cite{Bagchi16Tight3Mflds}). These two types of tight triangulations correspond to the cases $k=1$ and $k=d/2$ of Theorem 13 in \cite{Bagchi15TightnessCrit}.

The following two statements characterize these two classes of tight triangulations in terms of their $f$-vector (and the homology of the manifold), providing further justification for \Cref{conj:LutzKuehnel}:
  
\begin{theorem}[K\"uhnel \cite{Kuehnel95TightPolySubm} for $k=2$; Novik and Swartz \protect{\cite[Theorem 4.4]{Novik08SocBuchsMod}} for higher $k$]
\label{thm:kuhnel}
Let $C$ be a triangulation of a $2k$-dimensional manifold. Then we have
\[
{f_0(C) -k-2\choose k+1} \ge (-1)^k \binom{2k+1}{k+1}(\chi(C) - 2),
\]
with equality if and only if $C$ is $(k+1)$-neighborly (and thus tight).
\end{theorem}

\Cref{thm:kuhnel} was first conjectured as \cite[Conjecture B]{Kuehnel95TightPolySubm}.

\begin{theorem}[Murai \protect{\cite[Thm. 5.3(i)]{Murai15GradedBettiNumbers}}]
\label{thm:lss}
Let $C$ be a triangulation of a normal pseudo-manifold with $n$ vertices and of dimension $d\geq 3$. Then we have
\[
g_2 \ge { d+2 \choose 2} \beta_1(C, \F).
\]
with equality if and only if $C$ is stacked.
\end{theorem}

In \Cref{thm:lss}, the case of equality in dimension $\geq 4$ is due to Novik and Swartz \cite{Novik08SocBuchsMod}, and in dimension $d=3$ to Bagchi \cite{Bha2016} (confirming \cite[Problem 5.3]{Novik08SocBuchsMod}).

Naturally, for any triangulation we have that ${h_1 \choose 2}  \geq g_2$, with equality being equivalent to $2$-neighborliness. Hence, triangulations achieving the equality ${h_1 \choose 2}  = {d+2 \choose 2} \beta_1(C, \F)$ are the ones that are both stacked and $2$-neighborly. These are called {\em tight-neighborly} and were conjectured to be $\F$-tight by Lutz, Sulanke and Swartz~\cite{Lutz08FVec3Mnf}. This conjecture was confirmed in dimension three by Burton, Datta, Singh and Spreer \cite{Burton14SepIndex2Spheres} and in $d\ge 4$ by Effenberger \cite{Effenberger09StackPolyTightTrigMnf}. 

\begin{corollary}
\label{cor:stackedtype}
Let $C$ be a triangulation of an $\F$-orientable $d$-manifold with $n$ vertices, $d\geq 3$. Then we have
\[
{n - d- 1 \choose 2} = {h_1 \choose 2}  \ge {d+2 \choose 2} \beta_1(C, \F),
\] 
with equality if and only if $C$ is tight-neighborly (which implies $\F$-tightness).
\end{corollary}

An infinite family of tight triangulations of stacked type is described by K\"uhnel in \cite{Kuehnel95TightPolySubm}. Many more such examples -- including an infinite family as well as numerous sporadic examples -- are constructed in \cite{BDSSConstructionTight} by a systematic search using a set of conditions first formulated in \cite{Datta12InfFamTightTrig}. 

Tight triangulations of stacked and neighborly types are opposite in the following sense. Assume the manifold satisfies Poincar\'e duality, that is, $\F$ is arbitrary if $\manifold$ is orientable and it is of characteristic two if not.
In the neighborly case, the underlying $d$-manifold $\manifold$ must be $(d/2-1)$-connected and thus the only non-zero Betti numbers are $\tilde \beta_0(\manifold)=\beta_{2k}(\manifold)=1$ and  $\beta_k = (-1)^k (\chi(C) - 2)$. 
In the stacked case, on the other hand, we must have all Betti numbers equal to zero except 
$\tilde \beta_0(\manifold)=\beta_{2k}(\manifold)=1$ and  $\beta_1 = \beta_{d-1}$.
Nonetheless, both cases can be unified by \Cref{thm:bagchi-intro}, which has been the initial motivation to study the $\sep$- and $\mu$-vectors of a simplicial complex.

\medskip

Since in \Cref{thm:bagchi-intro} the Betti numbers of a tight triangulation must attain an upper bound in form of its $\mu$-vector, and since tight triangulations are conjectured to be minimal (see \Cref{conj:LutzKuehnel}), one might conjecture that the vertex links of all tight combinatorial $d$-manifolds must have $\sep$-vectors of Billera-Lee spheres with matching $f$-vectors. However, this is not the case, as the following example shows. Tight triangulations do not always have vertex links with maximal $\sep$-vector entries. 
This fact is somewhat surprising in light of \Cref{conj:LutzKuehnel}.

\begin{example}
\label{ex:tight}
Consider the $15$-vertex triangulation of $(S^1\!\!\sim\!\!S^3) \# (\mathbb{C}P^2)^{\# 5}$ constructed in~\cite{Kuehnel99CensusTight} (here, $\sim$ denotes the twisted product). It has Betti numbers over the field with two elements $\mathbb{F}_2$ of $\tilde{\beta}_1 = 1= \tilde{\beta}_3$ and $\tilde{\beta}_2 =5$. Its automorphism group is transitive on its vertices and its vertex links have $f$-vector $(1,14,64,100,50)$  and  $\sep$-vector  $(1/15, 1/15, 1/3, 1/15,1/15)$. It follows that the $\mu$-vector equals $(1,1,5,1,1)$ and hence the triangulation is $\mathbb{F}_2$-tight due to \Cref{thm:bagchi-intro}.

However, the Billera-Lee sphere with $f$-vector $(1,14,64,100,50)$ has $\sep_0 = 71/792$ which is considerably higher than $1/15$. \end{example}

\subsection{Triangulated $4$-manifolds with transitive automorphism group}
\label{ssec:4mflds}

\Cref{conj:general-intro} and \Cref{thm:upperbound} provide upper bounds for all entries of the $\sep$-vector in terms of the $f$-vector of a triangulated sphere.

For $d=3$, let $\sphere$ and $\altsphere$ be $n$-vertex $3$-spheres, $\sphere$ stacked and $\altsphere$ $2$-neighborly. Then we have that $\sep_1 (\sphere) = 0 = \sep_0 (\altsphere)$, and thus $\sep_0 (\sphere) = {n-4 \choose 2}/(n+1){6 \choose 2}$ and $\sep_1 (\altsphere) = 2{n-3 \choose 3}/(n+1){6 \choose 3}$. This has the following implications for triangulated $4$-manifolds:

Let $\manifold$ be an $(n+1)$-vertex, $2$-neighborly $4$-manifold with $m$ triangles. A simple calculation shows that then 
$$ \frac{2}{3}(2n^2-3n-5) \leq m \leq {n+1 \choose 3} .$$
Whenever the lower bound is satisfied, all vertex links are $(n-1)$-vertex stacked $3$-spheres and we necessarily have $\mu_1(\manifold) = {n-4 \choose 2}/{6 \choose 2}$ and $\mu_2 = 0$. Similarly, whenever the upper bound is attained, all vertex links are $(n-1)$-vertex $2$-neighborly $3$-spheres, and we have $\mu_1(\manifold) = 0$ and $\mu_2 = 2{n-3 \choose 3}/{6 \choose 3}$. By \Cref{cor:3d} this transforms into
\[
 \tilde\beta_1 (\manifold) \leq {n-4 \choose 2}/{6 \choose 2}; \quad \tilde\beta_2 (\manifold) = 0 
\]
for a $2$-neighborly $4$-manifold $\manifold$ with only stacked vertex links, and 
\[
 \tilde\beta_1 (\manifold') = 0; \quad \tilde\beta_2 (\manifold') \leq  2{n-3 \choose 3}/{6 \choose 3} 
 \]
for a $3$-neighborly $4$-manifold $\manifold'$. Both bounds coincide with existing bounds on $4$-manifolds, see \Cref{thm:kuhnel} and \Cref{cor:stackedtype}. Moreover, equality in these bounds here always implies by \Cref{thm:bagchi-intro} that the triangulation is both minimal and tight. Note, however, that bounds obtained this way rely on the triangulated $4$-manifold to have a prescribed $f$-vector of a certain kind. 

Here we want to generalize these results using the upper bound from \Cref{conj:general-intro}. To keep calculations simple we only consider the case where all vertex links have the same $f$-vector (as is the case for $2$-neighborly and stacked, as well as for 3-neighborly triangulations, as explained above). This also includes the case when $\manifold$ has vertex-transitive automorphism group (see \cite{Lutz99TrigMnfFewVertGeom3Mnf} for a classification of such triangulations for small numbers of vertices).

\medskip

\begin{figure}[htb]
  \begin{tabular}{r|r|r|r|r|l}
    \toprule
    $f(\lk)$	& $f(\manifold)$	& $\chi$	& Thm.~\ref{conj:general-intro} & Triangulations\\
    		& 			& 			& $(\tilde\beta_1,\tilde\beta_2) \leq$	&  & \\
    \midrule

$(5,10)$ &	$(6,15,20,15,6)$	& $2$	&$(0.00,0.00)$ & $\partial \Delta_4$\\ 
$(8,28)$ &	$(9,36,84,90,36)$	& $3$	&$(0.00,1.00)$ &$(\mathbb{C}P^2)_9$ \cite{Kuehnel83The9VertComplProjPlane}\\ 
$(9,36)$ &	$(10,45,120,135,54)$	& $4$	&$(0.00,2.00)$ & Does not exist \cite{Kuehnel83Uniq3Nb4MnfFewVert}\\ 
$(10,30)$ &	$(11,55,110,110,44)$	& $0$	&$(1.00,0.00)$ & $(S^1 \times S^3)_{11}$ \cite{Kuehnel95TightPolySubm}\\ 
$(11,36)$ &	$(12,66,144,150,60)$	& $0$	&$(1.17,0.34)$ &\\ 
$(11,41)$ &	$(12,66,164,180,72)$	& $2$	&$(0.75,1.5)$ &\\ 
$(11,46)$ &	$(12,66,184,210,84)$	& $4$	&$(0.41,2.82)$ &\\ 
$(11,51)$ &	$(12,66,204,240,96)$	& $6$	&$(0.12,4.24)$ &\\ 
$(13,48)$ &	$(14,91,224,245,98)$	& $0$	&$(1.77,1.55)$ &\\ 
$(13,63)$ &	$(14,91,294,350,140)$	& $7$	&$(0.76,6.53)$ &\\ 
$(13,78)$ &	$(14,91,364,455,182)$	& $14$	&$(0.00,12.00)$ &\\ 
$(14,46)$ &	$(15,105,230,240,96)$	& $-4$	&$(3.00,0.00)$ &\\ 
$(14,50)$ &	$(15,105,250,270,108)$	& $-2$	&$(2.48,0.96)$ &\\ 
$(14,52)$ &	$(15,105,260,285,114)$	& $-1$	&$(2.33,1.66)$ &\\ 
$(14,54)$ &	$(15,105,270,300,120)$	& $0$	&$(2.10,2.20)$ &\\ 
$(14,56)$ &	$(15,105,280,315,126)$	& $1$	&$(2.00,3.00)$ &\\ 
$(14,60)$ &	$(15,105,300,345,138)$	& $3$	&$(1.66,4.33)$ &\\ 
$(14,62)$ &	$(15,105,310,360,144)$	& $4$	&$(1.49,4.98)$ &\\ 
$(14,64)$ &	$(15,105,320,375,150)$	& $5$	&$(1.34,5.68)$ &$((S^1\!\!\sim\!\!S^3)\# 5(\mathbb{C}P^2))_{15}$ \cite{Kuehnel99CensusTight}\\ 
$(14,66)$ &	$(15,105,330,390,156)$	& $6$	&$(1.27,6.54)$ &\\ 
$(14,68)$ &	$(15,105,340,405,162)$	& $7$	&$(1.10,7.21)$ &\\ 
$(14,70)$ &	$(15,105,350,420,168)$	& $8$	&$(0.95,7.91)$ &\\ 
$(14,72)$ &	$(15,105,360,435,174)$	& $9$	&$(0.88,8.77)$ &\\ 
$(14,74)$ &	$(15,105,370,450,180)$	& $10$	&$(0.84,9.69)$ &\\ 
$(14,76)$ &	$(15,105,380,465,186)$	& $11$	&$(0.61,10.22)$ &\\ 
$(14,78)$ &	$(15,105,390,480,192)$	& $12$	&$(0.51,11.02)$ &\\ 
$(14,80)$ &	$(15,105,400,495,198)$	& $13$	&$(0.45,11.91)$ &\\ 
$(14,82)$ &	$(15,105,410,510,204)$	& $14$	&$(0.42,12.85)$ &\\ 
$(14,84)$ &	$(15,105,420,525,210)$	& $15$	&$(0.19,13.39)$ &\\ 
$(14,86)$ &	$(15,105,430,540,216)$	& $16$	&$(0.09,14.19)$ &\\ 
$(14,88)$ &	$(15,105,440,555,222)$	& $17$	&$(0.04,15.08)$ &\\ 
$(14,90)$ &	$(15,105,450,570,228)$	& $18$	&$(0.01,16.02)$ &\\ 
$(15,60)$ &	$(16,120,320,360,144)$	& $0$	&$(2.60,3.20)$ &\\ 
$(15,75)$ &	$(16,120,400,480,192)$	& $8$	&$(1.44,8.88)$ &\\ 
$(15,105)$ &	$(16,120,560,720,288)$	& $24$	&$(0.00,22.00)$&$K3_{16}$ \cite{Casella01TrigK3MinNumVert} \\

    \bottomrule
  \end{tabular}
  \caption{Upper bound on Betti numbers coming from \Cref{conj:general-intro} for $2$-neighborly triangulated $4$-manifolds with vertex links of constant $f$-vectors. \label{fig:upperbound}}
\end{figure}

Suppose that $\manifold$ is an orientable connected $4$-manifold with Betti numbers $\tilde\beta_0 +1= \tilde\beta_4 = 1$, $\tilde\beta_1 = \tilde\beta_3 = k$ and $\tilde\beta_2 = \ell$ and with vertex set $V$. A simple calculation shows that $f_i(\manifold)$ satisfies the following identity:
\[
f_i (\manifold) = \frac{1}{i+1} \sum \limits_{v \in V} f_{i-1} (\lk_v (\manifold)); \qquad \qquad 0\leq i \leq 4.
\]

Assuming that all links have equal $f$-vector $f (\lk (\manifold)) = (f_{-1}, \ldots , f_{3})$ this simplifies to
\[
f_i (\manifold)/f_0(\manifold) = \frac{f_{i-1}}{i+1} ; \qquad \qquad 0\leq i \leq 4.
\]
Since the links of $\manifold$ are $3$-spheres we have $f_2 = 2(f_1-f_0)$ and $f_3 = f_1 - f_0$ and we set $n:=f_0$ and $e:=f_1$.

Now combining the inequality $\tilde\beta_i (\manifold) \leq \mu(\manifold)$ with the definition of the $\mu$-vector and the upper bound from \Cref{conj:general-intro} we have the following statements.

{\small
\begin{align}
        \label{eq:beta1}
	\tilde\beta_1 (\manifold) \leq&\,\, \mu_1 (M) \nonumber \\
	=&\,\, f_0(\manifold)  \cdot \sep_0(\lk (\manifold)) \nonumber \\
        \leq &\,\, f_0(\manifold) \cdot\sep_0 (G(n,e,d)) \nonumber \\
        =&\,\, f_0 (\manifold) \left ( \frac{1}{n+1} - \frac{1}{k+1} + \frac{1}{(j+1)(j+2)} + \frac{n-k-1}{(d+2)(d+3)}\right )
\end{align}
}%
where $G(n,e,d)$ denotes the Billera-Lee graph, $k, j \in \mathbb{Z}$ are given by the fact that $k$ is the largest integer such that $e = j + {k \choose 2} + \frac{n-k-1}{d+1}$. By \Cref{cor:3d} this yields:

{\small
\begin{align}
        \label{eq:beta2}
	\tilde\beta_2 (\manifold) \leq&\,\, \mu_2 (M) \nonumber \\
	=&\,\, f_0(\manifold) \cdot \sep_1(\lk (\manifold)) \nonumber \\
        \leq & \,\,f_0(\manifold) \cdot\sep_1 (G(n,e,d)) \nonumber \\
        =&\,\, f_0 (\manifold) \left ( 2 \sep_0 (G(n,e,d)) - \frac{n^2 - 4n + 5}{5(n+1)} + \frac{e}{30} \right ) \nonumber \\
        =& \,f_0 (\manifold) \left ( 2 \left ( \frac{1}{n+1} - \frac{1}{k+1} + \frac{1}{(j+1)(j+2)} + \frac{n-k-1}{(d+2)(d+3)}\right ) - \frac{n^2 - 4n + 5}{5(n+1)} + \frac{e}{30} \right ).
\end{align}}

While these bounds are quite difficult to analyse by hand, they provide the basis for an algorithm to find a lower bound on the number of faces necessary to triangulate a combinatorial $4$-manifold $\manifold$ with fixed first and second Betti numbers (here subject to the additional condition that all vertex links in the triangulation have the same $f$-vector): 
\begin{enumerate}
  \item Given $\tilde\beta_1(\manifold)$ and $\tilde\beta_2(\manifold)$, go through all theoretically possible $f$-vectors of a triangulated $4$-manifold of Euler characteristic $\chi(\manifold) = 2- 2\tilde\beta_1 + \tilde\beta_2$ (with all vertex links sharing the same $f$-vector) in lexicographically increasing order. 
  \item For each such $f$-vector, check whether the upper bound on $\mu_1$ and $\mu_2$ (\Cref{eq:beta1,eq:beta2}) attains or exceeds $\tilde\beta_1(\manifold)$ and $\tilde\beta_2(\manifold)$ componentwise.
  \item The first $f$-vector satisfying this condition acts as a lower bound to triangulate~$\manifold$.
\end{enumerate}

In the case of $2$-neighborly triangulations of $4$-manifolds (i.e., the only case where triangulations are not trivially non-tight) feasible $f$-vectors together with their Euler characteristic and an upper bound on their first and second Betti numbers are listed in \Cref{fig:upperbound}.

\bibliographystyle{plain}

\end{document}